\documentclass[10pt,reqno]{article}

\usepackage{graphicx} 
\usepackage{amsmath, amssymb,epsfig,bm}
\usepackage{latexsym}
\usepackage{algorithmicx}
\usepackage[ruled,vlined,linesnumbered]{algorithm2e}
\usepackage{mathtools}
\mathtoolsset{showonlyrefs}
\usepackage{graphicx,subfigure}
\usepackage{color}
\usepackage{appendix}
\numberwithin{equation}{section}

\newtheorem{lemma}{Lemma}
\newtheorem{prop}[lemma]{Proposition}

\newtheorem{theorem}[lemma]{Theorem}

\newcommand{\be}{\begin{eqnarray}}
	\newcommand{\ee}{\end{eqnarray}}
\newcommand{\beq}{\begin{equation}}
	\newcommand{\eeq}{\end{equation}}
\newcommand{\ben}{\begin{eqnarray*}}
	\newcommand{\een}{\end{eqnarray*}}

\numberwithin{lemma}{section}

\begin{document}
	
	\title{A novel model reduction method to solve inverse problems of parabolic type}

	\author{Wenlong Zhang\thanks{Corresponding author. Department of Mathematics $\&$  National Center for Applied Mathematics
			Shenzhen, Southern University of Science and Technology (SUSTech), 1088 Xueyuan Boulevard, University Town of Shenzhen, Xili, Nanshan, Shenzhen, Guangdong Province, P.R.China. (zhangwl@sustech.edu.cn).
		}
		\and
		Zhiwen Zhang\thanks{Corresponding author. Department of Mathematics, The University of Hong Kong, Pokfulam Road, Hong Kong SAR, P.R.China. (zhangzw@hku.hk).
	}}
	\date{}
	\maketitle

	\begin{abstract}
		
		In this paper, we propose novel proper orthogonal decomposition (POD)--based model reduction methods that effectively address the issue of inverse crime in solving parabolic inverse problems. Both the inverse initial value problems and inverse source problems are studied. By leveraging the inherent low-dimensional structures present in the data, our approach enables a reduction in the forward model complexity without compromising the accuracy of the inverse problem solution. Besides, we prove the convergence analysis of the proposed methods for solving parabolic inverse problems. Through extensive experimentation and comparative analysis, we demonstrate the effectiveness of our method in overcoming inverse crime and achieving improved inverse problem solutions. The proposed POD model reduction method offers a promising direction for improving the reliability and applicability of inverse problem-solving techniques in various domains.

	\end{abstract}
	
	
	
	\section{Introduction}
	
	Inverse crime, the phenomenon where the forward model used for solving an inverse problem is the same as the one used for generating the data, poses a significant challenge in accurate and reliable inverse problem solutions. 
	
	Inverse problems arise in various fields of science and engineering, ranging from medical imaging and geophysics to material science and finance. Inverse problems require the estimation of an unknown parameter or field of interest from indirect measurements, which are often noisy and incomplete. The solution of inverse problems is challenging due to the ill-posedness of the problem, which leads to unstable and non-unique solutions. To overcome these challenges, various regularization techniques have been proposed to impose constraints on the solution space. However, the accuracy and reliability of inverse problem solutions can be significantly impacted by inverse crime. 
	
	Inverse crime refers to a situation where the forward model used for generating the data is the same as the one used for solving the inverse problem. This scenario leads to overly optimistic results and underestimates the uncertainties associated with the solution. Inverse crime can be a significant issue in practical applications, where the forward model is often an approximation of the underlying physical system and contains modeling errors and uncertainties. 
	
	To overcome the issue of inverse crime, we propose a novel Proper Orthogonal Decomposition (POD) model reduction method for solving inverse problems. The POD method is a data-driven technique that enables the identification of the dominant modes of variability in the data and the construction of a low-dimensional representation of the data. By leveraging the inherent low-dimensional structures present in the data, the POD method enables the reduction of the forward model complexity without compromising the accuracy of the inverse problem solution. 
	
	In this paper, we outline our new POD model reduction method for solving inverse problems and demonstrate its effectiveness in overcoming inverse crime. We first introduce the basic principles of the POD method and its application in inverse problems. We then present our new method for addressing the issue of inverse crime by incorporating the POD method into the inverse problem solution process. We demonstrate the performance of our method through extensive experimentation and comparative analysis with state-of-the-art methods. The results show that our proposed POD model reduction method outperforms existing methods in terms of accuracy and reliability, and offers a promising avenue for enhancing the applicability of inverse problem-solving techniques in various domains.

	One of the successful model reduction ideas in solving time-evolution problems is the proper orthogonal decomposition (POD) method \cite{sirovich1987,berkooz1993POD}. The POD method uses the data from an experiment or an accurate numerical simulation and extracts the most energetic modes in the system by using the singular value decomposition. This approach generates low-dimensional structures that can approximate the solutions to the time-evolution problem with high accuracy. The POD method has been used to solve many types of PDEs, including linear parabolic equations \cite{volkwein2013proper,kunisch2001galerkin},  Navier‐Stokes equations \cite{kunisch2001galerkin}, viscous G-equations \cite{gu2021error},  Hamilton–Jacobi–Bellman (HJB) equations \cite{kunisch2004hjb}, and optimal control problems \cite{alla2013time}. The interested reader is referred to \cite{quarteroni2015reduced,Willcox2015PODsurvey,hesthaven2016certified} for a comprehensive introduction to the model reduction methods. 
	
	In this paper, we will develop a novel POD method to solve the forward and inverse problems of the parabolic type.

	To start with, we consider a parabolic equation as follows:
	
	\begin{equation}\label{parabolic-equation}
		\left\{
		\begin{aligned}
			u_t +\mathcal{L}u &= f(x) &\mbox{in } \Omega\times (0, T), \\
			u(x, t)&= 0  &\mbox{on } \partial \Omega\times (0, T),\\
			u(x, 0)&= g(x) &\mbox{in } \Omega\,,
		\end{aligned} 
		\right.
	\end{equation}
	where $\Omega\subset \mathbb R^d$ $(d=1,2,3)$ is a bounded domain with a $C^2$ boundary or a convex domain satisfying the uniform cone condition, $\mathcal{L}$ denotes a second-order elliptic operator given by $\mathcal{L}u=-\nabla\cdot (q(x)\nabla u) +c(x)u$, and $g(x)$ is the initial condition. We assume the elliptic operator $\mathcal{L}$ is uniform elliptic, i.e., there exist $q_{\min}, a_{\max}>0$ such that $q_{\min}<q(x)<q_{\max}$ for all $x \in \Omega$. Additionally, we assume $q(x)\in C^{1}(\bar{\Omega})$, $c(x)\in C(\bar\Omega)$ and $c(x)\geq 0$. 
	
	Let $u$ represent the solution of the parabolic equation \eqref{parabolic-equation}. We define the forward operator $\mathcal{S}:$ $\mathcal{S}(f,g)=u(\cdot,T)$. The forward problem involves computing the solution $u(\cdot,t)$ for $t>0$ given the source term $f(x)$ and initial condition $g(x)$. The inverse problem, on the other hand, aims to reconstruct $f(x)$ or $g(x)$ from the final time measurement $m=u(\cdot,T)$. Typically, iterative methods are employed to solve the inverse problem. During each iterative step, one may need to solve the forward problem one or more times. Consequently, the majority of computations expenses are attributed to the computation of the forward problems. 
	
	In this paper, we will solve two types of inverse problems:
	\begin{enumerate}
		\item Inverse source problem: recover the source term $f(x)$ using the final time measurement $m=u(\cdot,T)$ and the known initial term $g(x)$. 
		\item Backward problem: recover the initial term $g(x)$ using the final time measurement $m=u(\cdot,T)$ and the known source term $f(x)$. 
	\end{enumerate}
	
Iterative methods are usually used to solve the inverse problems. For each iterative step, one may have to solve the forward problem one or more times, thus most of the computations are costed by the computation of the forward problem.

To solve the inverse problem in a faster way, the authors construct the POD basis functions from the snapshot solutions of the parabolic equation \eqref{parabolic-equation} with fixed source functions in \cite{ZhangJCP2023}. The proposed method accelerates the computation of the inverse source problem, yet   In this paper, we develop a novel POD method to solve the forward and inverse problems of the parabolic type. We will give a brief review of the traditional POD method in the appendix \ref{appendix-POD}, including the construction of the POD basis functions. 

The rest of the paper is organized as follows. In Section 2, we introduce the Ajoint-POD method for solving parabolic inverse source problems and provide the error estimate for the proposed methods. Similarly, in Section 3, we propose the Ajoint-POD method for solving parabolic backward problem and provide the corresponding error estimate. In Section 4, we present numerical results to demonstrate the accuracy of our method. Finally, concluding remarks are made in Section 5.

\section{Ajoint-POD method for parabolic inverse source problems}\label{sec:source}

	The traditional POD method has a drawback: to construct the POD basis functions, one needs to know the source term $f(x)$ or the initial term $g(x)$ in advance. However, in inverse problems, the source term or the initial term is precisely what we want to find. This can lead to the so-called inverse crime, which should be avoided in practice. In \cite{ZhangJCP2023}, the authors studied this issue by assuming that the true source term belongs to a known function class, thus avoiding the inverse crime. However, this approach does not completely address the issue of the inverse crime. 
	
	To tackle this challenge, we propose a novel model reduction method for this type of inverse problem: the Adjoint-POD method in this paper. Our new method efficiently solves inverse problems without requiring a priori information about the source term or initial term. By combining the Adjoint method's strengths with the POD method's model reduction capabilities, the Adjoint-POD method can efficiently and quickly solve inverse problems while avoiding the inverse crime issue.

\subsection{Ajoint POD method}\label{sec-adj}
	
	To demonstrate the idea of the Adjoint-POD method, we will first apply it to solve the inverse source problem. For the inverse source problem of the parabolic equation, the objective is to recover the unknown source term $f(x)$, given the final time measurement $m(x)=\mathcal{S}(f)=u(\cdot,T)$. In this case, $u$ satisfies the following equation:

	\begin{equation}\label{source}
		\left\{
		\begin{aligned}
			u_t +\mathcal{L}u &= f(x) &\mbox{in } \Omega\times (0, T), \\
			u(x, t)&= 0  &\mbox{on } \partial \Omega\times (0, T),\\
			u(x, 0)&= 0 &\mbox{in } \Omega\,.
		\end{aligned} 
		\right.
	\end{equation}
	
	Here we assume $u(x, t)= 0$ and $u(x, 0)= 0$ for simplicity, otherwise, one just need to subtract the background solution  from the measurement $m(x)$. Since the source term $f(x)$ is unknown, we cannot use the traditional POD method to obtain snapshots. Instead, we will acquire the snapshots from the following adjoint equation:
	\begin{equation}\label{source-adjoint}
		\left\{
		\begin{aligned}
			\Tilde{u}_t +\mathcal{L}\Tilde{u} &= m(x) &\mbox{in } \Omega\times (0, T), \\
			\Tilde{u}(x, t)&= 0  &\mbox{on } \partial \Omega\times (0, T),\\
			\Tilde{u}(x, 0)&= 0 &\mbox{in } \Omega\,.
		\end{aligned} 
		\right.
	\end{equation}
	
	Denote the snapshots 
	$\tilde y_k = \Tilde{u}(\cdot, t_{k - 1})$, $k = 1, \ldots, M + 1$  with $ M=\frac{T}{\Delta t}$, and $\tilde y_k = \overline{\partial} \Tilde{u}(\cdot, t_{k - M - 1})$, $k = M + 2, \ldots, 2m + 1$ with $\overline{\partial} \Tilde{u}(\cdot, t_{k}) = \frac{\Tilde{u}(\cdot, t_{k}) - \Tilde{u}(\cdot, t_{k - 1})}{\Delta t}$, $k = 1, \ldots, M$. Then we construct the new POD basis $\{\psi_1,...,\psi_{N_{\text{pod}}}\}$ using the method described in Appendix \ref{appendix-POD}  from the adjoint equation \eqref{source-adjoint}. 
 Denote $V_{\text{POD}}=span\{\psi_1,...,\psi_{N_{\text{pod}}}\}$

	We consider using these new POD basis functions $\{\psi_1,...,\psi_{N_{\text{pod}}}\}$ to approximate the forward problem to accelerate the computation. The fully discrete scheme is constructed on $V_{\text{pod}}$ and the solution is denoted by $U_k$ for $k=1\cdots M$. To be precise, we seek numerical solutions $U_k$'s such that
	
\begin{equation}\label{eqn:fully_discrete}
		(\bar{\partial}U_k,\psi)+a(U_k,\psi)=(f,\psi), \quad \forall \psi\in {V_{\text{pod}}}.
	\end{equation}
 Here the bilinear form $a(u,v)=(q\nabla u,\nabla v) +(cu,v)$. 
	We define the solution operator from the source term $f$ to the final time solution $U_M$ as $\mathcal{S}_{\text{pod}}$, i.e., $\mathcal{S}_{\text{pod}}f=U_M$. Using the new POD basis functions and the reduced-order model represented by $\mathcal{S}_{\text{pod}}$, we can efficiently solve the forward problem for each time step, significantly reducing the computational cost compared to the full-scale model. This approach is particularly useful when solving inverse problems, where multiple forward problem evaluations are required.

	\subsection{Convergence of the Adjoint-POD method} \label{sec:source-POD}
	
	We will first revisit an important property of the eigenvalue distribution for the classical elliptic operator $\mathcal{L}$ \cite{agmon,Fleckinger}. 
	\begin{prop}\label{para-lem:2.1} Suppose $\Omega$ is a bounded domain in $\mathbb{R}^d$ and $a(x), c(x)\in C^0(\bar{\Omega})$, $c(x)\geq 0$, then the eigenvalue problem
		\begin{align}\label{yy2}
			\mathcal{L}\psi =\mu\,\psi~~ \text{with} ~~ \psi_{\partial\Omega}=0
		\end{align}
		has a countable set of positive eigenvalues $\mu_1\le\mu_2\le\cdots$,  with its corresponding eigenfunctions 
		$\{\phi_k\}_{k=1}^\infty$ forming an orthogonal basis of $L^2(\Omega)$. 
		Moreover, there exist constants $C_1,C_2>0$ such that 
		$C_1 k^{2/d}\le \mu_k\le C_2k^{2/d}$ for all $k=1,2,\cdots.$
	\end{prop} 
	
	From the Proposition above, eigenfunction set $\{\phi_k\}_{k=1}^\infty$ forms an orthogonal basis of $L^2(\Omega)$. Then for any $f\in L^2(\Omega)$, we write $f=\sum_{k=1}^\infty f_k \phi_k$ for a set of coefficients $f_k$. 
	Let $u=\sum_{k=1}^\infty u_k(t) \phi_k$ be the solution of the problem \eqref{source}.
	Substituting these two expressions of $f$ and $u$ into the first equation of \eqref{source}, 
	we get by noting the fact that $\mathcal{L}\phi_k=\mu_k\phi_k$ and comparing the coefficients of $\phi_k$ 
	on both sides of the equation that $u_k(0)= 0$ and 
	\begin{equation}\label{ODEs1}
		u'_k(t) + \mu_k u_k=f_k \quad \quad \text{in} ~(0, T)\,. 
	\end{equation} 
	This equation expresses the time evolution of the coefficients $u_k(t)$ in terms of the coefficients $f_k$ of the source term $f$. We can write the solution as $u_k(T)=\alpha_k\,f_k$, with $\alpha_k=e^{-\mu_kT}\int_0^Te^{\mu_ks}ds=\frac{1}{\mu_k}(1-e^{-\mu_k T})$.  Noting that $Sf=u(\cdot, T)=\sum^\infty_{k=1} u_k(T) \phi_k$, we can formally write 
	\begin{equation*}
		S\Big(\sum_{k=1}^\infty f_k \phi_k\Big)=  \sum_{k=1}^\infty \alpha_k f_k \phi_k.
	\end{equation*}
	This representation of the solution operator $S$ provides a convenient way to compute the solution $u(\cdot, T)$ using the eigenfunctions $\phi_k$ and the coefficients $\alpha_k$. For simplicity, we approximate the source term $f(x)$ by a finite-dimensional truncation, i.e. \begin{equation}\label{truncation-source}
		f_{\text{app}}=\sum_{k=1}^L f_k \phi_k.
	\end{equation}
	Then, the solution $u(x,t)$ of the parabolic equation has the form: 
     \begin{equation}
     u(\cdot, T)=\sum^L_{k=1} \frac{1}{\mu_k}(1-e^{-\mu_k T})f_k \phi_k. 
 \end{equation}
 After simple calculation, we will also derive that 
 \begin{equation}
 \Tilde{u}(x,t)=\sum^L_{k=1} \frac{1}{\mu_k}(1-e^{-\mu_k T})(1-e^{-\mu_k t})f_k \phi_k.
 \end{equation}
  Actually the POD basis \eqref{PODbasisMethodOfSnapshot} is nothing but the singular value decomposition of the matrix $\tilde A=(\tilde y_1,...,\tilde y_M)$, where $\tilde y_j=(\Tilde{u}(x_1,t_j),...,\Tilde{u}(x_N,t_j))^T$. Here $x_1,...,x_N$ are the finite element nodes in $\Omega$. Suppose $A$ has the singular value decomposition: $A=U\Sigma V$, then $\psi_ks$ are exactly the first $M$ columns of $U$. 
	
	Let us denote $A=(y_1,...,y_M)$ and $\Tilde{A}=(\tilde y_1,...,\tilde y_M)$, the matrix $\Phi=(\phi_1(\vec x,t_1),...,\phi_L(\vec x,t_1))$, $F=\text{diag}(f_1,...,f_L)$ and $D=\text{diag}(\frac{1}{\mu_k}(1-e^{-\mu_k T}),...,\frac{1}{\mu_L}(1-e^{-\mu_L T}))$. Additionally, let us define an $L\times M$ matrix $J$ with entries  $J(i,j)=\frac{1}{\mu_i}(1-e^{-\mu_i t_j})$. $\Phi$ is a column orthogonal matrix due to the normal orthogonality of the eigenfunctions $\phi_k$. Utilizing the formulations of $u$ and $\tilde u$, we can represent the matrices $A$ and $\tilde A$ as follows:
	\begin{equation}
		A=\Phi F J, ~~\text{and}~~ \tilde A=\Phi D F J.
	\end{equation}
	Proposition \ref{Prop_PODError} demonstrates that the low-rank space $V_{\text{pod}}$ is the best $N_{\text{pod}}$-rank approximation of the column space of $\tilde A$. Our objective is to show that $V_{\text{pod}}$ is also a good approximation of the column space of $A$, which will validate the efficiency of the new POD method. To begin, let us 
	establish the relationship between the matrices $A$ and $\tilde A$.

	\begin{lemma}\label{lemma-span}
		If $L\leq M$, then $span\{y_1,...,y_M\}=span\{\tilde y_1,...,\tilde y_M\}$, i.e. $C(A)=C(\tilde A)$.
	\end{lemma}
	\begin{proof}
		Here we provide a concise proof for the case when the eigenvalues $\mu_j$ are distinct from each other.
		The proof for the case of repeated eigenvalues follows a similar approach. To demonstrate the desired results, 
		we need to show that the existence of matrices $P$ and $\tilde P$ such that,
		$$\Phi D F J P=\Phi F J,$$
		$$\Phi D F J=\Phi F J \tilde P.$$
		We will only present a brief proof for the first equality, as the second can be derived in a similar manner. 
		
		Since the columns of $\Phi$ are independent and the diagonal matrix $F$ is invertible, it suffices to prove
		the existence of a matrix $P$ such that 
		$$J P=D J.$$
		
		First, we show that the matrix $J_{L\times L}'$ with entries $J'(i,j)=1-e^{-\mu_i t_j}$ is invertible. 
		Denote the vector $\textbf{e}=(1,...,1)^T$. Then, we can express $J'$ as 
		$$J'=\textbf{e}\textbf{e}^{T}-V_L,$$ 
		where $V_L(i,j)=e^{-\mu_i t_j}$ is a Vandermonde matrix. To prove the invertibility of $J'$, we assume, 
		by contradiction, that $J'$ is singular. In that case, there exists a nonzero vector $\textbf{c}=(c_1,...,c_L)^T$ such that 
		$$J'\textbf{c}=0,$$ 
		or equivalently, 
		$$V_L\textbf{c}=\textbf{e}\textbf{e}^{T}\textbf{c}.$$
		Now, consider the function 	$f(x)=\sum_{j=1}^L c_je^{xt_j}$. 
		Under this assumption, we have that 
		$$f(0)=f(\mu_1)=f(\mu_2)=\cdots=f(\mu_L),$$
		which implies that the function $f$ has $L+1$ distinct zeros. This implies that the derivative $f'(x)=\sum_{j=1}^L c_jt_je^{xt_j}$ has $L$ distinct zeros. Since $\textbf{c}$ is a nonzero vector and all $\mu_j$s are all nonzero, this will
		imply that the Vandermonde matrix $V_L$ is singular, which contradicts the fact that $V_L$ is an invertible matrix. Consequently, $J'$ must be a nonsingular matrix. 
		
		Since the invertibility of $J'$, the first $L$ columns of $J$ are independent and thus form a basis for $R^L$. Similarly, the matrix $DJ$ also has independent columns that form a basis for $R^L$. Consequently, there must 
		exist a matrix $P$ such that $$J P=D J.$$
		This result establishes that the spans of the sets $\{y_1,...,y_M\}$ and $\{\tilde y_1,...,\tilde y_M\}$ are equivalent, i.e., $\operatorname{span}\{y_1, \dots, y_M\} = \operatorname{span}\{\tilde y_1, \dots, \tilde y_M\}$. This ends the proof.
	\end{proof}
	
	Based on Proposition \ref{Prop_PODError}, the new POD basis effectively approximates the set $\{\tilde y_1,...,\tilde y_M\}$. Given the previous results, we can now demonstrate that the new POD basis also serves as a good approximation 
	for the original set $\{ y_1,..., y_M\}$.
	\begin{theorem}\label{theorem-pod-app}
		Using the same notation as in Proposition \ref{Prop_PODError}, if a sufficient number of snapshots are available, i.e. $L\leq M$, then the following approximation error bound holds: 
		\begin{equation}\label{appro-adj}		
			\frac{\sum_{i=1}^{M }\left|\left| y_i - P_{\text{pod}}y_i\right|\right|_{L^2(\Omega)}^{2}}{		 \sum_{i=1}^{M}\left|\left| y_i\right|\right|_{L^2(\Omega)}^{2}} 
			\leq C L^{4/d} \rho,
		\end{equation} 
		where $P_{\text{pod}}$ is the projection operator onto the adjoint-POD space $\operatorname{span}\{\psi_1, \dots, \psi_{N_{\text{pod}}}\}$ and  $\rho=\frac{\sum_{k={N_{\text{pod}}}+1}^{2M + 1}  \lambda_k}{\sum_{k=1}^{2M + 1}\lambda_k}$.
	\end{theorem}
	\begin{proof}
		In the following proof, we assume $L=M$ for simplicity. For the case $L<M$, the proof is similar. 
		
		Using the same notation of Lemma \ref{lemma-span}, $\Phi$ and $J$ are both invertible square matrices. Then there exists a unique matrix $P$ such that,
		$$\Phi D F J P=\Phi F J,$$
		and $P=J^{-1}D^{-1}J$. 
		
		Hence $y_j=\sum^L_{i=1}P_{ij}\tilde y_i$. Then by using the Cauchy-Schwarz inequality, for any $1\leq j\leq L$, we have that 
		\begin{align}
			\|y_j-P_{\text{pod}}y_j\|^2\leq \sum_{i=1}^L P^2_{ij}\sum_{i=1}^L\|\tilde y_i-P_{\text{pod}} \tilde y_i\|^2.
		\end{align}
		Hence,
		\begin{align}
			\sum_{j=1}^L\|y_j-P_{\text{pod}}y_j\|^2\leq \sum_{i,j=1}^L P^2_{ij}\sum_{i=1}^L\|\tilde y_i-P_{\text{pod}} \tilde y_i\|^2=\|P\|^2_F\sum_{i=1}^L\|\tilde y_i-P_{\text{pod}} \tilde y_i\|^2.
		\end{align}
		The rest is to estimate the Frobenius norm of $P$. Since $P=J^{-1}D^{-1}J$, we define $\|P\|_d=\|D^{-1}\|_2$. It is easy to verify that $\|\cdot\|_d$ is a matrix norm. Then,
		$$\|P\|_F\leq C \|P\|_d= C \|D^{-1}\|_2\leq C \mu_L.$$
		
		On the other hand, since $\Phi$ is an orthogonal matrix, we have,
		\begin{align}
			\sum_{j=1}^L\|\tilde y_j\|^2&=\|\Phi D F J \|_F^2\\
			&=\|D F J\|_F^2 \leq \|D \|_F^2\|F J\|_F^2\\
			&\leq  C  \|F J\|_F^2 \\
			&\leq  C \sum_{j=1}^L\|y_j\|^2.
		\end{align}
		
		Alighed with Proposition \ref{Prop_PODError}, we finally have,
		\begin{align}
			\frac{\sum_{i=1}^{M}\left|\left| y_i - P_{\text{pod}}y_i\right|\right|_{L^2(\Omega)}^{2}}{		 \sum_{i=1}^{M}\left|\left| y_i\right|\right|_{L^2(\Omega)}^{2}} &\leq C \|P\|^2_F \frac{\sum_{i=1}^{M}\left|\left| \tilde y_i - P_{\text{pod}}\tilde y_i\right|\right|_{L^2(\Omega)}^{2}}{		 \sum_{i=1}^{M}\left|\left| \tilde y_i\right|\right|_{L^2(\Omega)}^{2}}\\
			&\leq \mu_L^2 \rho.
		\end{align}
		
		The conclusion comes with the estimation $\mu_i\leq C i^{2/d}$. 
	\end{proof}
	
	\subsection{Convergence of inverse parabolic source problem}\label{sec:convergence-source}
	
	To solve this inverse source problem, we use the well-established Tikhonov regularization method, expressed as  
	\begin{align}\label{general-regularization}
		\mathop {\rm min}\limits_{f \in X}         \|\mathcal{S}(f)-m\|_{L^2(\Omega)}^2+\lambda \|f\|_{L^2(\Omega)}^2.
	\end{align}
	
 However, in the conventional application of the POD method, the source term $f$ and the initial condition $g$ 
must be determined initially to generate snapshots and obtain the POD basis functions. In the context of inverse problems, the only available information is the measurement $m(x)$. This predicament, referred to as the inverse crime, makes this method impossible to implement in practice. Our new method could overcome this vital drawback by setting the forward solver to be our new POD forward solver.
	
	In the general discrete approximation of problem \eqref{general-regularization}, we seek to solve the following least-squares regularized optimization problem:
	\begin{align}\label{disc}
		\mathop {\rm min}\limits_{f\in V_{\text{pod}}} \|\mathcal{S}_{\text{pod}}(f)-m\|_{L^2(\Omega)}^2+\lambda \|f\|_{L^2(\Omega)}^2.
	\end{align}

	Consider the functional $\mathcal{J}_{\text{pod}}[f]=\|\mathcal{S}_{\text{pod}}f-m\|_{L^2(\Omega)}^2+\lambda \|f\|_{L^2(\Omega)}^2$. 
	By computing the Fr$\acute{e}$chet derivative of $\mathcal{J}_{\text{pod}}[f]$, we can derive the subsequent iterative scheme:
	\begin{align}
		f_{k+1}=f_{k}-\beta d\mathcal{J}_{\text{pod}}[f_{k}], \quad \forall k\in \mathbb{N},
	\end{align}
	where $\beta$ is the step size, $d\mathcal{J}_{\text{pod}}[f]=\mathcal{S}_{\text{pod}}^*(\mathcal{S}_{\text{pod}}f-m)+ \lambda f$ denotes the Fr$\acute{e}$chet derivative, and $f_{0}$ is an initial guess \cite{ZhangJCP2023}. 
	
	The above theory is based on noise-free case, i.e. the final time measurement $m=u(\cdot,T)$ is assumed to be precisely known. 
	However, in practical applications, measurement data often contains uncertainties.  We assume the measurement data is blurred by noise and takes the discrete form $$m^n_i=u(d_i,T)+e_i, i=1, \cdots, n,$$ where $d_i$s represent the positions of detectors and $\{e_i\}^n_{i=1}$ are independent and identically distributed (i.i.d.) random variables on an appropriate probability space ($\mathfrak{X},\mathcal{F},\mathbb{P})$. 
	Based on \cite{Chen-Zhang2021} and the analysis therein, we know that $\|u\|_{C([0,T];H^2(\Omega))}\le C\|f\|_{L^2(\Omega)}$. According to the embedding theorem of Sobolev spaces, we know that $H^2(\Omega)$ is continuously embedded into $C(\bar\Omega)$ so that  $u(\cdot,T)$ is well defined point-wisely for all $d_i\in \Omega$. Without loss of generality, we assume that the scattered locations $\{d_i\}_{i=1}^n$ are uniformly distributed in $\Omega$. That is, there exists a constant $B>0$ such that ${d_{\max}}/{d_{\min}} \leq B$, where ${d_{\max}}$ and ${d_{\min}}$ are defined by 
	\begin{align}\label{aa}
		d_{\max}=\mathop {\rm sup}\limits_{x\in \Omega} \mathop {\rm inf}\limits_{1 \leq i \leq n} |x-d_i|  
		~~~\mbox{and} ~~ ~
		d_{\min}=\mathop {\rm inf}\limits_{1 \leq i \neq j \leq n} |d_i-d_j|.
	\end{align}
	
	We will first use the technique developed in \cite{Chen-Zhang} to recover the final time measurement $u(\cdot,T)$ from the noisy data $m^n_i$ for $i=1,...,n$. We approximate $u(\cdot,T)$ by solving the following minimization problem:
	\begin{align}\label{denoise}
		m=\mathop {\rm argmin}\limits_{u \in X} \frac{1}{n}\sum_{i=1}^n(u(x_i)-m^n_i)^2+\alpha |u|_{H^2(\Omega)}^2.
	\end{align}
	
	Assume the pointwise noise $e_i$ has a bounded variance $\sigma$, which is referred to as the noise level. \cite{Chen-Zhang} analyzed this problem and provided optimal convergence results. Moreover, they proposed an a posteriori algorithm to obtain the best approximation without knowing the true solution $m$ and noise level $\sigma$. Here, we list their main results. If one chooses the optimal regularization parameter
	$$\alpha^{1/2+d/8}=O(\sigma n^{-1/2} \|u(\cdot,T)\|^{-1}_{H^2(\Omega)}),$$
	then the solution $m$ of \eqref{denoise} achieves the optimal convergence
	\begin{equation}\label{error-denoise}
		\mathbb{E}\big[ \|u(\cdot,T)-m\|_{L^2(\Omega)}\big] \leq C \alpha^{1/2}\|u(\cdot,T)\|_{H^2(\Omega)}.
	\end{equation} 
	
	And if the noise $\{e_i\}^n_{i=1}$ are independent Gaussian random variables with variance $\sigma$, we further have, 
	
	\begin{equation}
		\mathbb{P}( \|u(\cdot,T)-m\|_{L^2(\Omega)} \geq \alpha^{1/2}\|u(\cdot,T)\|_{H^2(\Omega)}z) \leq 2e^{-Cz^2}.
	\end{equation} 
	
	Using this recovered function $m(x)$, we generate the adjoint POD basis functions in Section \ref{sec-adj}. It can be easily shown that, with uncertainty, the POD basis functions are still good low-rank approximation of the snapshots $\{y_1,...,y_M\}$. Combining the Theorem \ref{theorem-pod-app} and \eqref{error-denoise}, we shall have that for any $1\leq i \leq M$, 
	
	\begin{equation}\label{appro-adj-noise}		
		\left|\left| y_i - P_{\text{pod}}y_i\right|\right|_{L^2(\Omega)}^{2}
		\leq C (M L^{4/d} \rho +\alpha) \|f\|^2_{L^2(\Omega)},
	\end{equation} 
	
	Since we replace the source term by a finite truncation \eqref{truncation-source}, and if $f\in H^1(\Omega)$, 
	
	\begin{equation}
		\|f-f_{\text{app}}\|_{L^2}\leq C \frac{\|\nabla f\|_{L^2}}{\sqrt{\mu_L}} \leq C \frac{\|\nabla f\|_{L^2}}{{L^{1/d}}}. 
	\end{equation}
	
	If $f\in L^2(\Omega)$, then $f_{\text{app}}\rightarrow f$ as $L\rightarrow +\infty$. We assume 
	\begin{equation}
		\|f-f_{app}\|^2_{L^2}\leq \varepsilon,
	\end{equation}
 where $\varepsilon$ depends on $L$. 
	With those results, using a similar technique to prove the Theorem 4.1 in \cite{ZhangJCP2023}, we have the following convergence results. 
	\begin{theorem}
		Let $\{e_i\}^n_{i=1}$ be independent random variables satisfying $\mathbb{E}[e_i]=0$ and $\mathbb{E}[e^2_i]\leq \sigma^2$ for $i=1,\cdots, n$. Set $\alpha^{1/2+d/8}=O(\sigma n^{-1/2} \|u^*(\cdot,T)\|^{-1}_{H^2(\Omega)})$ in \eqref{denoise}, and $\lambda =O(M L^{4/d} \rho +\alpha)$, then 
		\begin{align}\label{pod1}
			\mathbb{E}\big[\|\mathcal{S}f^*- \mathcal{S}_{\text{pod}}f_{\text{pod}}\|_{L^2(\Omega)}^2\big]\leq C\lambda\|f^*\|^2_{L^2(\Omega)} +C\varepsilon,
		\end{align}
		\begin{align}
			\mathbb{E}\big[\|f^*- f_{\text{pod}}\|_{L^2(\Omega)}^2\big]\leq C\|f^*\|^2_{L^2(\Omega)}+C\varepsilon,
		\end{align}
		and
		\begin{align}
			\mathbb{E}\big[\|f^*- f_{\text{pod}}\|_{H^{-1}(\Omega)}^2\big]&\leq C\lambda^{1/2}\|f^*\|^2_{L^2(\Omega)}+C\varepsilon.
		\end{align}
	\end{theorem}
	
	Furthermore, if we assume the noise $\{e_i\}^n_{i=1}$ are independent Gaussian random variables with variance $\sigma$, we will have a stronger type of convergence, one can refer to \cite{Chen-Zhang2021} for a similar proof. We just list the results here.
	
	\begin{theorem}
		Let $\{e_i\}^n_{i=1}$ be independent Gaussian random variables with variance $\sigma$. Set $\alpha^{1/2+d/8}=O(\sigma n^{-1/2} \|u^*(\cdot,T)\|^{-1}_{H^2(\Omega)})$ in \eqref{denoise}, and $\lambda =O(M L^{4/d} \rho +\alpha)$, then there exists a constant C, for any $z>0$, 
		\begin{align}
			\mathbb{P}(\|S_{\text{pod}}f_{\text{pod}}-Sf^*\|_{L^2(\Omega)}\geq(\lambda^{1/2}\|f^*\|_{L^2}+\varepsilon)z)\leq 2e^{-Cz^2},
		\end{align}
		\begin{align}
			\mathbb{P}(\|f_{\text{pod}}-f^*\|_{L^2(\Omega)}\geq(\|f^*\|_{L^2}+\varepsilon)z)\leq 2e^{-Cz^2},
		\end{align} 
		and
		\begin{align}
			\mathbb{P}(\|f_{\text{pod}}-f^*\|_{H^{-1}(\Omega)}\geq(\lambda^{1/4}\|f^*\|_{L^2}+\varepsilon)z)\leq 2e^{-Cz^2}.
		\end{align}
	\end{theorem}

	\section{Parabolic backward problem}\label{sec:backward}
	
	For the backward problem of the parabolic equation, our goal is to recover the initial term $g(x)$, given the final time measurement $m=\mathcal{S}(g)=u(\cdot,T)$. In this case, $u$ satisfies the following equation: 
	\begin{equation}\label{backward}
		\left\{
		\begin{aligned}
			u_t +\mathcal{L}u &= 0 &\mbox{in } \Omega\times (0, T), \\
			u(x, t)&= 0  &\mbox{on } \partial \Omega\times (0, T),\\
			u(x, 0)&= g(x) &\mbox{in } \Omega\,,
		\end{aligned} 
		\right.
	\end{equation}
	
	Unlike the traditional POD method, we will get the snapshots from the following adjoint equation:
	\begin{equation}\label{backward-adjoint}
		\left\{
		\begin{aligned}
			\tilde u_t +\mathcal{L}\tilde u &= 0 &\mbox{in } \Omega\times (0, T), \\
			\tilde u(x, t)&= 0  &\mbox{on } \partial \Omega\times (0, T),\\
			\tilde u(x, 0)&= m(x) &\mbox{in } \Omega\,,
		\end{aligned} 
		\right.
	\end{equation}
	In this case, the snapshots are generated by solving the adjoint equation \eqref{backward-adjoint} with the given final time measurement $m(x)$ as the initial condition. 

    Repeat the standard procedure in section \ref{sec:source-POD},   
	we generate the new POD basis $\psi_k$s from the snapshots $\big\{ \Tilde{u}(\cdot, t_0), \Tilde{u}(\cdot, t_1), \ldots, \Tilde{u}(\cdot, t_M) \big\}$, where $t_k = k \Delta t$ with $\Delta t = \frac{T}{M}$ and $k = 0, \ldots, M$. Then we have the following error formula similar to Proposition \ref{Prop_PODError}:
	\begin{equation}\label{ini-SnapshotOfSolutions}		
		\frac{\sum_{i=1}^{2M + 1}\left|\left|\tilde y_i - \sum_{k=1}^{{N_{\text{pod}}}}\big(\tilde y_i,\psi_k(\cdot)\big)_{L^2(\Omega}\psi_k(\cdot)\right|\right|_{L^2(\Omega)}^{2}}{		 \sum_{i=1}^{2M+1}\left|\left|\tilde y_i\right|\right|_{L^2(\Omega)}^{2}} 
		= \rho,
	\end{equation} 
	where the number $N_{\text{pod}}$ is determined according to the decay of the ratio $\rho=\frac{\sum_{k={N_{\text{pod}}}+1}^{2M + 1}  \lambda_k}{\sum_{k=1}^{2M + 1}\lambda_k}$.

	We consider using these new POD basis functions to approximate the forward problem to accelerate the computation. The fully discrete scheme is constructed on $V_{\text{pod}}$ and the solution is denoted by $U_k$ for $k=1\cdots M$ with $ M=\frac{T}{\Delta t}$. To be precise, we seek numerical solutions $U_k$'s such that
	
	\begin{equation}\label{eqn:fully_discrete-ini}
		(\bar{\partial}U_k,\psi)+a(U_k,\psi)=0, \quad \forall \psi\in {V_{pod}},
	\end{equation}
	with $U_0=g(x)$.
	We define the solution operator from the ini term $g$ to the final time solution $U_M$ as $\mathcal{S}_{\text{pod}}$, i.e., $\mathcal{S}_{\text{pod}}g=U_M$.

	\subsection{Convergence of the adjoint-POD method}
	
	For any $g\in L^2(\Omega)$, we write $g=\sum_{k=1}^\infty g_k \phi_k$ for a set of coefficients $g_k$. Let $u=\sum_{k=1}^\infty u_k(t) \phi_k$ be the solution of the problem \eqref{source}. Substituting these two expressions of $g$ and $u$ into the first equation of \eqref{source}, we get by noting the fact that $L\phi_k=\mu_k\phi_k$ and comparing the coefficients of $\phi_k$ on both sides of the equation that $u_k(0)= g_k$ and 
	\begin{equation*}
		u'_k(t) + \mu_k u_k=0\ \ \ \ \mbox{in } ~(0, T)\,.
	\end{equation*}
	We can write the solution as $u_k(T)=\alpha_k\,g_k$, with $\alpha_k=e^{-\mu_kT}$.  Noting that $Sg=u(\cdot, T)=\sum^\infty_{k=1} u_k(T) \phi_k$, we can formally write 
	\begin{equation*}
		S\Big(\sum_{k=1}^\infty g_k \phi_k\Big)=  \sum_{k=1}^\infty \alpha_k g_k \phi_k.
	\end{equation*} 
	This representation shows the relationship between the initial condition $g$ and the solution $u(\cdot, T)$ at the final time $T$. The operator $S$ maps the initial condition to the solution at time $T$ through the coefficients $\alpha_k$, which depend on the eigenvalues $\mu_k$ of the operator $L$ and the final time $T$. This relationship can be used to analyze the properties of the solution and the backward problem.
	
	For simplicity, we approximate the source term $g(x)$ by a finite-dimensional truncation, i.e. \begin{equation}\label{truncation-source2}
		g_{\text{app}}=\sum_{k=1}^L g_k \phi_k.
	\end{equation}
	
	Then the solution $u(x,t)$ of the parabolic equation has the form: $u(\cdot, T)=\sum^L_{k=1} e^{-\mu_kT}f_k \phi_k$. After simple calculation, we can also have that $\Tilde{u}(x,t)=\sum^L_{k=1} e^{-\mu_kT}e^{-\mu_kt}f_k \phi_k$. Choosing the POD basis is to compute the  singular value decompositio of the matrix $\tilde A=(\tilde y_1,...,\tilde y_M)$, where $\tilde y_j=(\Tilde{u}(x_1,t_j),...,\Tilde{u}(x_N,t_j))^T$. Here $x_1,...,x_N$ are the finite element nodes in $\Omega$. Suppose $A$ has the singular value decomposition: $A=U\Sigma V$, then the $\psi_ks$ are exactly the first $M$ columns of $U$. 
	
	Denote $A=(y_1,...,y_M)$ and $\Tilde{A}=(\tilde y_1,...,\tilde y_M)$, the matrix $\Phi=(\phi_1(\vec x,t_1),...,\phi_L(\vec x,t_1))$, $F=\text{diag}(f_1,...,f_L)$ and $D=\text{diag}(e^{-\mu_1T},...,e^{-\mu_LT})$, and the $L\times M$ matrix $J$ with entries: $J(i,j)=e^{-\mu_it_j}$. Obviously, $\Phi$ is a column orthogonal matrix since the normal orthogonality of eigenfunctions $\phi_k$s. With the formulations of $u$ and $\tilde u$, we could represent the matrix $A$ and $\tilde A$ by:
	\begin{equation}
		A=\Phi F J, \quad \quad  \text{and} \quad \quad \tilde A=\Phi D F J.
	\end{equation}
	
	These matrix representations of $A$ and $\tilde A$ provide a compact way to express the relationship between the coefficients of the eigenfunctions $\phi_k$ and the solutions $u(x,t)$ and $\tilde u(x,t)$.

	Proposition \ref{Prop_PODError} shows that the low-rank space $V_{\text{pod}}$ is the best $N_{\text{pod}}$ rank approximation of the column space of $\tilde A$, we want to show that $V_{pod}$ is also a good approximation of the column space of $A$, which will certify the efficiency of the new POD method. First of all, let us show the connection of the matrices $A$ and $\tilde A$.
	
	\begin{lemma}\label{lemma-span-ini}
		If $L\leq M$, then $span\{y_1,...,y_M\}=span\{\tilde y_1,...,\tilde y_M\}$. 
		This means the column spaces of $A$ and $\tilde A$ are identical, i.e., $C(A)=C(\tilde A)$.
	\end{lemma}
	\begin{proof}
		The proof is similar to the Lemma \ref{lemma-span}, just using the fact that the matrix $J$ is actually a Vandermonde matrix.
	\end{proof}
	
	
	From \eqref{ini-SnapshotOfSolutions}, the new POD basis is a good approximation of $\{\tilde y_1,...,\tilde y_M\}$. With the above preparation, we will show the new POD basis is also a good approximation of $\{ y_1,..., y_M\}$.
	\begin{theorem}\label{theorem-pod-app-ini}
		Using the same notation in this section,  if one has enough snapshots, i.e. $L\leq M$, then 
		\begin{equation}\label{appro-adj-ini}		
			\frac{\sum_{i=1}^{M}\left|\left| y_i - P_{pod}y_i\right|\right|_{L^2(\Omega)}^{2}}{		 \sum_{i=1}^{M}\left|\left| y_i\right|\right|_{L^2(\Omega)}^{2}} 
			\leq C e^{2\mu_LT} \rho,
		\end{equation} 
		where $P_{pod}$ is the projection operator on the adjoint-POD space $span\{\psi_1,...,\psi_{N_{pod}}\}$ and  $\rho=\frac{\sum_{k={N_{pod}}+1}^{M}  \lambda_k}{\sum_{k=1}^{M}\lambda_k}$
	\end{theorem}
	\begin{proof}
		In the following proof, we assume $L=M$ for simplicity. For the case $L<M$, the proof is similar. 
		
		Using the same notation of Lemma \ref{lemma-span-ini}, $\Phi$ and $J$ are both invertible square matrices. Then, there exists a unique matrix $P$ such that,
		$$\Phi D F J P=\Phi F J,$$
		and $P=J^{-1}D^{-1}J$. 
		
		Hence $$y_j=\sum^L_{i=1}P_{ij}\tilde y_i.$$
		
		Then by Cauchy-Schwarz inequality, for any $1\leq j\leq L$, we have the estimate 
		\begin{align}
			\|y_j-P_{\text{pod}}y_j\|^2\leq \sum_{i=1}^L P^2_{ij}\sum_{i=1}^L\|\tilde y_i-P_{\text{pod}} \tilde y_i\|^2.
		\end{align}
		Hence,
		\begin{align}
			\sum_{j=1}^L\|y_j-P_{\text{pod}}y_j\|^2\leq \sum_{i,j=1}^L P^2_{ij}\sum_{i=1}^L\|\tilde y_i-P_{\text{pod}} \tilde y_i\|^2=\|P\|^2_F\sum_{i=1}^L\|\tilde y_i-P_{\text{pod}} \tilde y_i\|^2.
		\end{align}
		The rest is to estimate the Frobenius norm of $P$. Since $P=J^{-1}D^{-1}J$, we define $\|P\|_d=\|D^{-1}\|_2$. It is easy to verify that $\|\cdot\|_d$ is a matrix norm. Then,
		$$\|P\|_F\leq C \|P\|_d= C \|D^{-1}\|_2\leq C e^{\mu_LT}.$$
		
		On the other hand, since $\Phi$ is an orthogonal matrix, we have 
		\begin{align}
			\sum_{j=1}^L\|\tilde y_j\|^2&=\|\Phi D F J \|_F^2=\|D F J\|_F^2 \leq \|D \|_F^2\|F J\|_F^2\\
			&\leq  C  \|F J\|_F^2 \\
			&= \leq  C \sum_{j=1}^L\|y_j\|^2.
		\end{align}
		
		Alighed with \eqref{ini-SnapshotOfSolutions}, we finally have,
		\begin{align}
			\frac{\sum_{i=1}^{M}\left|\left| y_i - P_{pod}y_i\right|\right|_{L^2(\Omega)}^{2}}{		 \sum_{i=1}^{M}\left|\left| y_i\right|\right|_{L^2(\Omega)}^{2}} &\leq C \|P\|^2_F \frac{\sum_{i=1}^{M}\left|\left| \tilde y_i - P_{pod}\tilde y_i\right|\right|_{L^2(\Omega)}^{2}}{		 \sum_{i=1}^{M}\left|\left| \tilde y_i\right|\right|_{L^2(\Omega)}^{2}}\\
			&\leq C e^{2\mu_LT} \rho.
		\end{align}
		
		The conclusion comes with the estimation $\mu_i\leq C i^{2/d}$. 
	\end{proof}
	
	\subsection{Convergence of backward problem} \label{sec:convergence-ini}
	
	To solve this inverse problem, we use the traditional Tikhonov regularization method, 
	\begin{align}\label{general-regularization-ini}
		\mathop {\rm min}\limits_{g \in X}         \|\mathcal{S}(g)-m\|_{L^2(\Omega)}^2+\lambda \|g\|_{L^2(\Omega)}^2.
	\end{align}
	
	In the general discrete approximation  of the problem \eqref{general-regularization-ini}, we solve the following least-squares regularized optimization problem:
	
	\begin{align}\label{disc-ini}
		\mathop {\rm min}\limits_{g\in V_{\text{pod}}}         \|\mathcal{S}_{\text{pod}}(g)-m\|_{L^2(\Omega)}^2+\lambda \|g\|_{L^2(\Omega)}^2.
	\end{align}
	
	But in the traditional setting of the POD method, one has to know the source term $f$ and the initial condition $g$ first to derive the snapshots to get the POD basis functions, but the only information known in inverse problems is the measurement $m(x)$. This is called inverse crime which makes this method impossible to implement in practice.

	Define the functional $\mathcal{J}_{\text{pod}}[g]=\|\mathcal{S}_{\text{pod}}g-m\|_{L^2(\Omega)}^2+\lambda \|g\|_{L^2(\Omega)}^2$. We can compute the 
	Fr$\acute{e}$chet derivative of $\mathcal{J}_{\text{pod}}[g]$ and obtain the following iterative scheme:
	
	\begin{align}
		g_{k+1}=g_{k}-\beta d\mathcal{J}_{pod}[g_{k}], \quad \forall k\in \mathbb{N},
	\end{align}
	where $\beta$ is the step size, $d\mathcal{J}_{\text{pod}}[g]=\mathcal{S}_{\text{pod}}^*(\mathcal{S}_{\text{pod}}g-m)+ \lambda f$, and $g_{0}$ is an initial guess. 
	
	The above theory is based on noise free case, i.e. the final time measurement $m=u(\cdot,T)$ is exactly known. For practical consideration, the measurement data always contains uncertainty.  We assume the measurement data is always blurred by noise and takes the discrete form $m^n_i=u(d_i,T)+e_i$, $i=1, \cdots, n$, where $d_i$s are the positions of detectors and $\{e_i\}^n_{i=1}$ are independent and identically distributed (i.i.d.) random variables on a proper probability space ($\mathfrak{X},\mathcal{F},\mathbb{P})$. From property of paraboic equation, we know that $\|u\|_{C([0,T];H^2(\Omega))}\le C\|g\|_{L^2(\Omega)}$. 
	According to the embedding theorem of Sobolev spaces, we know that $H^2(\Omega)$ is continuously embedded into $C(\bar\Omega)$ so that  $u(\cdot,T)$ is well defined point-wisely for all $d_i\in \Omega$. 
	
	We will first use the technique developed in \cite{Chen-Zhang} to recover the final time measurement $u(\cdot,T)$ from the noisy data $m^n_i$ for $i=1,...,n$. We approximate $u(\cdot,T)$ by solving the following minimization problem:
	
	\begin{align}\label{denoise-ini}
		m=\mathop {\rm argmin}\limits_{u \in X} \frac{1}{n}\sum_{i=1}^n(u(x_i)-m^n_i)^2+\alpha |u|_{H^2(\Omega)}^2.
	\end{align}
	
	Assume the point wise noise $e_i$ has bounded variance $\sigma$, this is the so-called noise level. \cite{Chen-Zhang}analyzed this problem and give the optimal convergence results, moreover they proposed a posteriori algorithm to give the best approximation without knowing the true solution $m$ and noise level $\sigma$. Here we list their main results:
	If one chooses the optimal regularization parameter
	$$\alpha^{1/2+d/8}=O(\sigma n^{-1/2} \|u(\cdot,T)\|^{-1}_{H^2(\Omega)}),$$
	then the solution $m$ of \eqref{denoise-ini} achieves the optimal convergence
	\begin{equation}\label{error-denoise-ini}
		\mathbb{E}\big[ \|u(\cdot,T)-m\|_{L^2(\Omega)}\big] \leq C \alpha^{1/2}\|u(\cdot,T)\|_{H^2(\Omega)}.
	\end{equation} 
	
	And if the noise $\{e_i\}^n_{i=1}$ are independent Gaussian random variables with variance $\sigma$, we further have, 
	
	\begin{equation}
		\mathbb{P}( \|u(\cdot,T)-m\|_{L^2(\Omega)} \geq \alpha^{1/2}\|u(\cdot,T)\|_{H^2(\Omega)}z) \leq 2e^{-Cz^2}.
	\end{equation} 
	
	Using this recovered function $m(x)$, we generate the adjoint POD basis in Section \ref{sec-adj}. It can be easily shown that, with uncertainty, the POD basis is still a good low rank approximation of the snapshots $\{y_1,...,y_M\}$. Combining the Theorem \ref{theorem-pod-app-ini} and \eqref{error-denoise-ini}, we shall have that for any $1\leq i \leq M$,

	\begin{equation}\label{appro-adj-noise2}		
		\left|\left| y_i - P_{\text{pod}}y_i\right|\right|_{L^2(\Omega)}^{2}
		\leq C (M e^{2\mu_L T} \rho +\alpha) \|g\|^2_{L^2(\Omega)},
	\end{equation} 
	
	Since we replace the source term by a finite truncation \eqref{truncation-source2}, we assume 
	\begin{equation}
		\|g-g_{\text{app}}\|^2_{L^2}\leq \varepsilon(L).
	\end{equation}
	With those results, using a similar technique to prove the Theorem 3.7 in \cite{zhang-zhang-arxiv2023}, we have the following convergence results. 
	\begin{theorem}
		Let $\{e_i\}^n_{i=1}$ be independent random variables satisfying $\mathbb{E}[e_i]=0$ and $\mathbb{E}[e^2_i]\leq \sigma^2$ for $i=1,\cdots, n$. Set $\alpha^{1/2+d/8}=O(\sigma n^{-1/2} \|u(\cdot,T)\|^{-1}_{H^2(\Omega)})$ in \eqref{denoise}, and $\lambda =O(M e^{2\mu_L T} \rho +\alpha)$, then 
		\begin{align}\label{pod1-2}
			\mathbb{E}\big[\|\mathcal{S}g^*- \mathcal{S}_{\text{pod}}g_{\text{pod}}\|_{L^2(\Omega)}^2\big]\leq C\lambda\|g^*\|^2_{L^2(\Omega)} +C\varepsilon,
		\end{align}
		
		\begin{align}
			\mathbb{E}\big[\|g^*- g_{\text{pod}}\|_{L^2(\Omega)}^2\big]\leq C\|g^*\|^2_{L^2(\Omega)}+C\varepsilon.
		\end{align}
		
	\end{theorem}
	
	Furthermore, if we assume the noise $\{e_i\}^n_{i=1}$ are independent Gaussian random variables with variance $\sigma$, we will have a stronger type of convergence, one can refer to \cite{zhang-zhang-arxiv2023} for a similar proof. We just list the results here.
	
	\begin{theorem}
		Let $\{e_i\}^n_{i=1}$ be independent Gaussian random variables with variance $\sigma$. Set $\alpha^{1/2+d/8}=O(\sigma n^{-1/2} \|u(\cdot,T)\|^{-1}_{H^2(\Omega)})$ in \eqref{denoise}, and $\lambda =O(M e^{2\mu_L T} \rho +\alpha)$, then there exists a constant C, for any $z>0$, 
		\begin{align}
			\mathbb{P}(\|S_{\text{pod}}g_{\text{pod}}-Sg^*\|_{L^2(\Omega)}\geq(\lambda^{1/2}\|g^*\|_{L^2}+\varepsilon)z)\leq 2e^{-Cz^2},
		\end{align}
		
		\begin{align}
			\mathbb{P}(\|g_{\text{pod}}-g^*\|_{L^2(\Omega)}\geq(\|g^*\|_{L^2}+\varepsilon)z)\leq 2e^{-Cz^2}.
		\end{align}
		
	\end{theorem}

\section{Numerical examples}
In this section, we present several numerical examples to demonstrate the reconstruction results for the inverse source problem and the backward problem discussed in this paper. We consider the domain $\Omega=[0,\pi]^2$. For each observation data set, we first apply the backward Euler scheme in time and the linear finite element method (specifically, the P1 element) in space with a mesh size of $h=1/50$ and a time step of $\Delta t=T/400$. We select 9 POD basis functions to compute the inverse problems, whereas, for the finite element method, there are approximately 2500 basis functions.

In \cite{ZhangJCP2023}, the authors have already compared the efficiency of the POD method and the finite element method for solving this inverse problem. They demonstrate that the POD method achieves a speed-up of at least 6 times, even with 400 finite element basis functions. Therefore, we do not include a comparison with the finite element method in this paper. As the number of finite element basis functions increases, the potential for the POD method to achieve greater speed-up also grows correspondingly.

 \subsection{Inverse source examples}
	
	In the following examples, we apply the adjoint POD method to recover the source term $f$ as described in Section \ref{sec:source}.  We obtain the data for the forward problem with the exact source term $f$ at the final time $T=1$.


    \textbf{Example 4.1} We first demonstrate the importance of choosing the appropriate POD basis functions. For the same source term, we apply different right-hand sides of equation \eqref{source} to obtain the POD basis functions. Subsequently, we solve the inverse source problem using these different POD basis functions. 
    
    Figure \ref{right POD basis} illustrates this process. The true source term is given by $\sin(2x)\sin(2y)e^{\frac{x+y}{\pi}}$, and its surface plot is shown in Figure \ref{right POD basis} (a). Figure \ref{right POD basis} (b) presents the reconstruction result obtained using the adjoint POD method proposed in this paper, indicating that our new method effectively recovers the source term in an efficient manner. Figure \ref{right POD basis} (c) displays the result when an incorrect right-hand side is used to generate the POD basis functions. In this case, we use $\sin(x)\sin(y)$ as the right-hand side in \eqref{source} to generate the POD basis functions. Figure \ref{right POD basis} (d) shows the result when we use an A-shaped function as the right-hand side to derive the POD basis. 

      As can be seen, Figure \ref{right POD basis} (b) provides a good reconstruction, whereas Figures \ref{right POD basis} (c) and (d) yield inaccurate results. The result in Figure \ref{right POD basis} (d) is particularly striking, as the recovered image deviates significantly from the exact source term.

      In this example, we demonstrate the importance of selecting the appropriate basis functions for solving inverse problems. Utilizing an unsuitable set of basis functions can lead to wrong results. Our proposed adjoint POD method offers a set of suitable basis functions for such problems. In the following examples, we will compare our adjoint POD basis functions with the original true POD basis to validate Theorem \ref{theorem-pod-app}.

	\begin{figure}[htb]
		\centering
		\subfigure[Exact source term]{\includegraphics[width=0.32\linewidth]{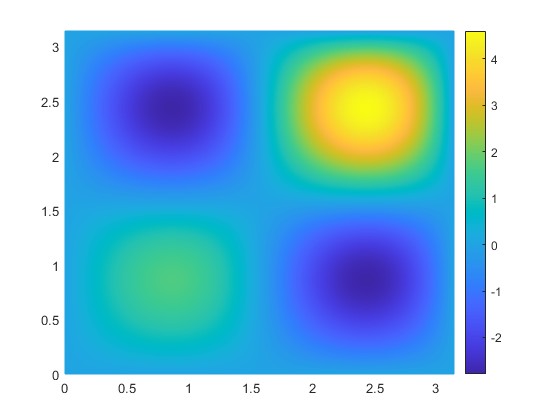}}
        \subfigure[Reconstructed by the Adjoint POD method]{\includegraphics[width=0.32\linewidth]{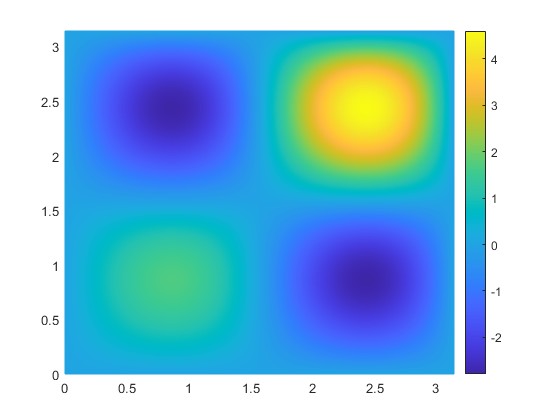}}\\
		\subfigure[Using the POD basis generated with the right hand side $\sin(x)\sin(y)$ ]{\includegraphics[width=0.32\linewidth]{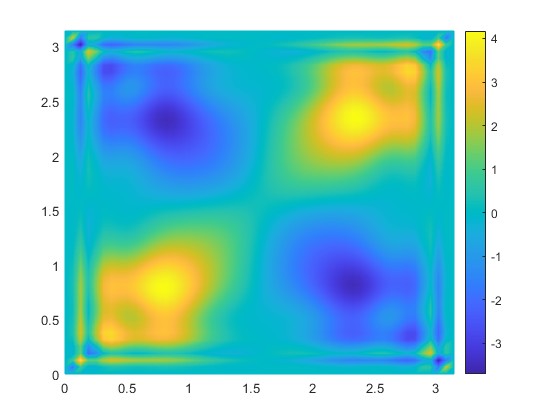}}	
		\subfigure[Using the POD basis generated with the right hand side of A shaped function ]{\includegraphics[width=0.32\linewidth]{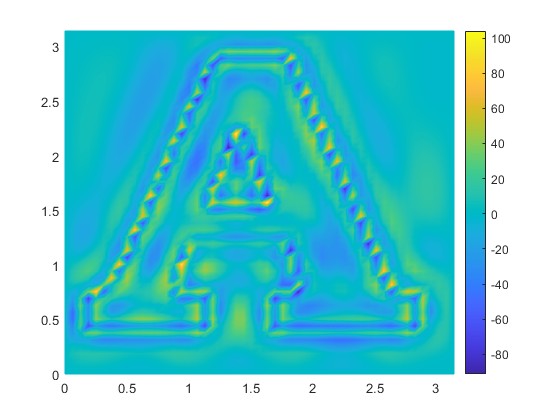}}
		\caption{The importance of choice of POD basis  }\label{right POD basis}
	\end{figure}



    \textbf{Example 4.2} In this example, we will first use the true source term as the right-hand side in \eqref{source} to generate the POD basis. Then, we will generate the POD basis using our proposed adjoint POD method. To validate Theorem \ref{theorem-pod-app}, we will compare both sets of basis functions and assess their similarity. Figure \ref{basis-comp-source-sin} shows the results when using the exact source term $f^*=\sin(2x)\sin(2y)$. Figure \ref{basis-comp-source-sin} (b) and Figure \ref{basis-comp-source-sin} (d) show that both the traditional POD and our adjoint POD work well to recover the true source term. However, our adjoint POD method does not require prior knowledge of the exact source term, while the traditional POD method does, leading to the so-called inverse crime. The basis functions for each method are depicted in Figure \ref{basis-comp-source-sin} (c) and Figure \ref{basis-comp-source-sin} (e), demonstrating that our adjoint POD basis is highly similar to the traditional one, even though we derived it solely from measured data. 
    
    Figure \ref{basis-comp-source-Z} presents the results when using an exact source term $f^*$ in the form of a Z-shaped function. This example also illustrates the efficiency of the POD method in solving inverse problems compared to the finite element method. Figure \ref{basis-comp-source-Z} (c) and Figure \ref{basis-comp-source-Z} (e) show that both the basis functions of the traditional POD method and our adjoint POD method contain the critical information of the exact source term that we aim to recover. In contrast, the basis functions of finite element method do not contain any prior information about the true function we need to recover.
 	
	\begin{figure}[htb]
		\flushright
		\subfigure[Exact source term]{\includegraphics[width=0.32\linewidth]{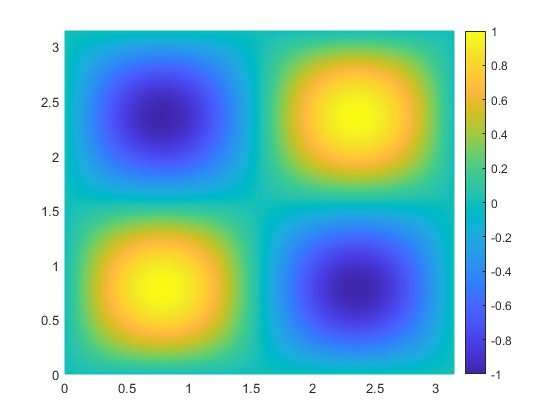}}
		\subfigure[Recovered result by traditional POD with exact right hand side ]{\includegraphics[width=0.32\linewidth]{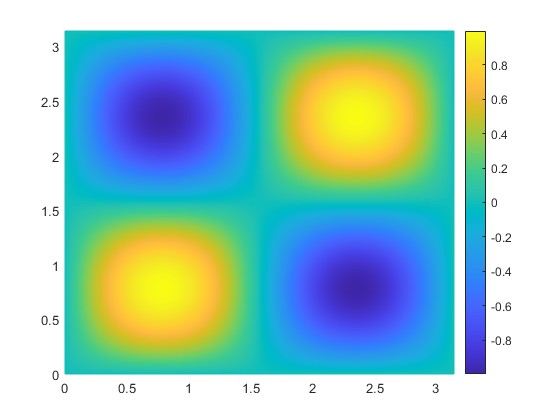}}	
		\subfigure[Basis of the traditional POD]{\includegraphics[width=0.32\linewidth]{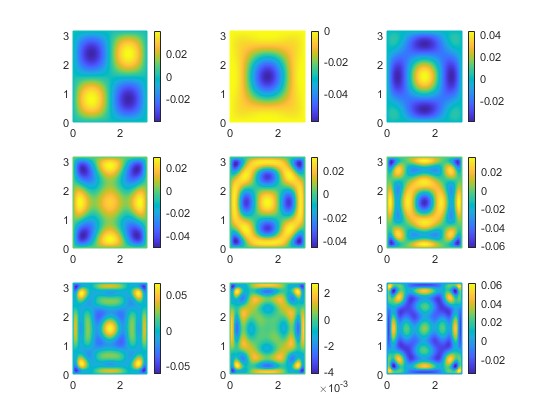}}
		\subfigure[Result by the adjoint POD]{\includegraphics[width=0.32\linewidth]{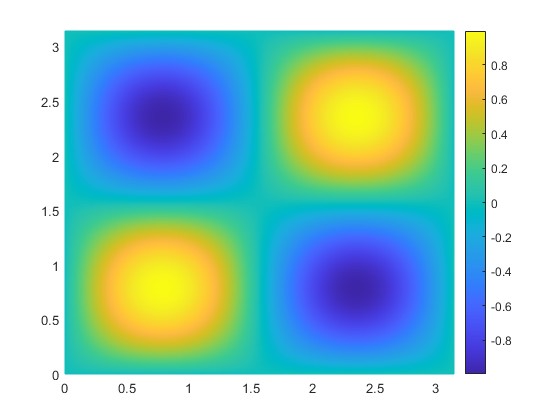}}	\subfigure[Basis of the adjoint POD]{\includegraphics[width=0.32\linewidth]{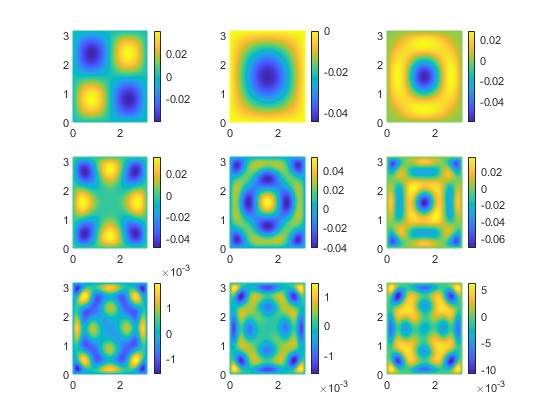}}
		\caption{ Comparison of basis between traditional POD and  the adjoint POD for $f=\sin(2x)\sin(2y)$. }\label{basis-comp-source-sin}
	\end{figure}

	\begin{figure}[htb]
		\flushright
		\subfigure[Exact source term]{\includegraphics[width=0.32\linewidth]{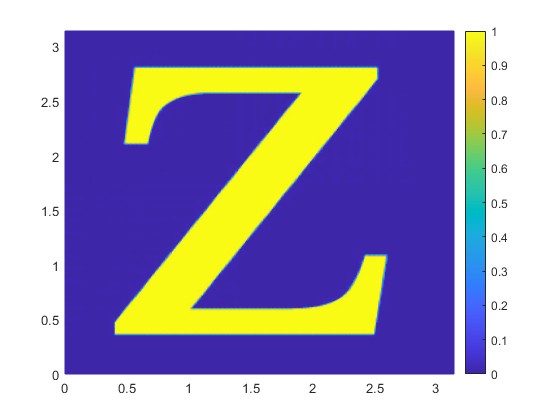}}
		\subfigure[Recovered result by traditional POD with exact right hand side ]{\includegraphics[width=0.32\linewidth]{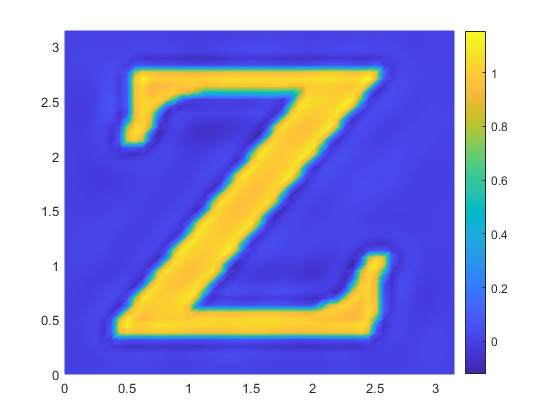}}	
		\subfigure[Basis of the traditional POD]{\includegraphics[width=0.32\linewidth]{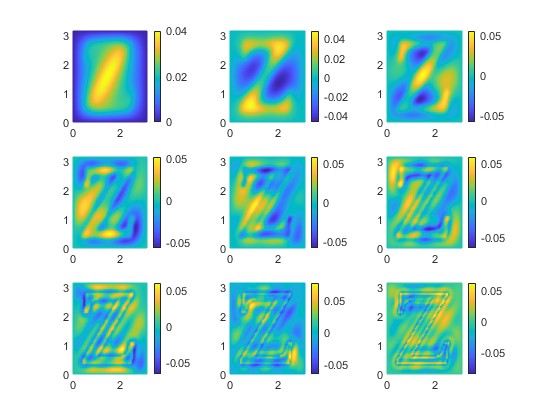}}
		\subfigure[Result by the adjoint POD]{\includegraphics[width=0.32\linewidth]{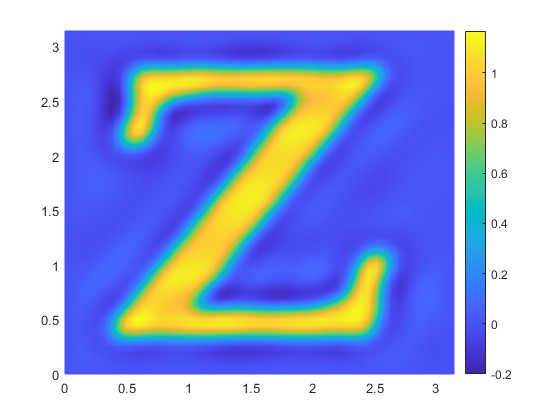}}	\subfigure[Basis of the adjoint POD]{\includegraphics[width=0.32\linewidth]{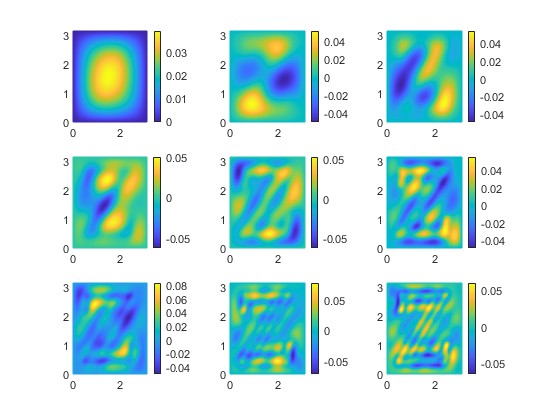}}
		\caption{ Comparison of basis between traditional POD and  the adjoint POD for $f^*$ of Z shaped function. }\label{basis-comp-source-Z}
	\end{figure}

        In the aforementioned cases, all the measured data were noise-free. We will now test the denoising method described in Section \ref{sec:convergence-source} by examining highly challenging cases with noise levels ranging from $10\%$ to $50\%$.

      \textbf{Example 4.3} In this example, we will evaluate the robustness of the adjoint POD method in the presence of noise. We consider the measurement data to be $m^n_i=u(d_i,T)+\sigma e_i$, $i=1, \cdots, n$, where $d_i$ represents positions within the domain $\Omega$, and $\{e_i\}^n_{i=1}$ are independent standard normal random variables. We will take 2500 positions $d_i$ uniformly distributed over the domain $\Omega$. 

      Figure \ref{sin2exp-noise-source} demonstrates the robustness of our method in the presence of significant noise. Even with a 50\% noise level, where the measured data is entirely obscured by noise as shown in Figure \ref{sin2exp-noise-source} (c), our method is still able to recover the source term as depicted in Figure \ref{sin2exp-noise-source} (f).

 \begin{figure}[htb]
		\centering
		\subfigure[Exact source term]{\includegraphics[width=0.32\linewidth]{Figures/source-sin2exp-true.jpg}}
        \subfigure[Measured data with 25$\%$ noise]{\includegraphics[width=0.32\linewidth]{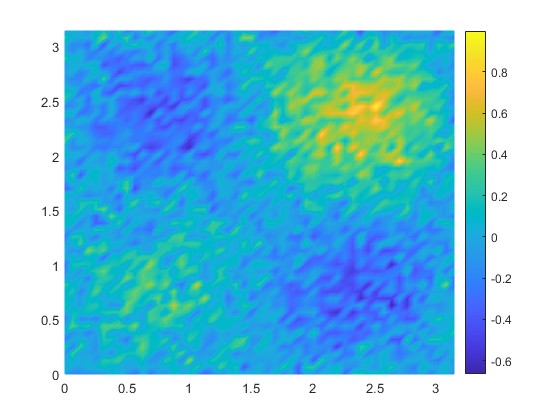}}	
        \subfigure[Measured data with 50$\%$ noise]{\includegraphics[width=0.32\linewidth]{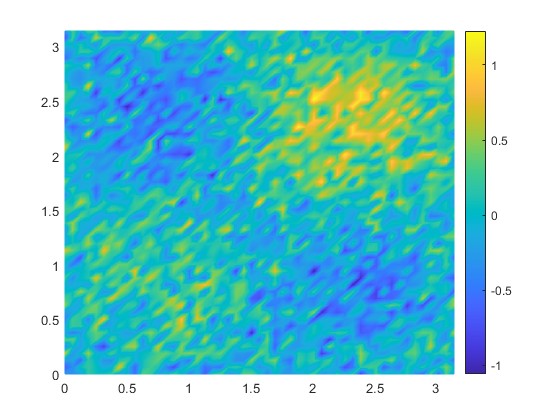}}
		\subfigure[Recovered result for 10$\%$ noise]{\includegraphics[width=0.32\linewidth]{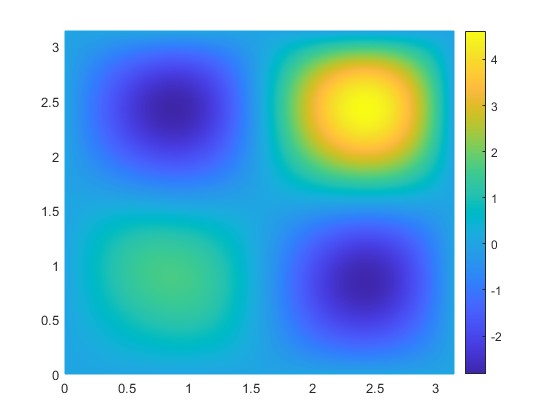}}
        \subfigure[Recovered result for 25$\%$ noise]{\includegraphics[width=0.32\linewidth]{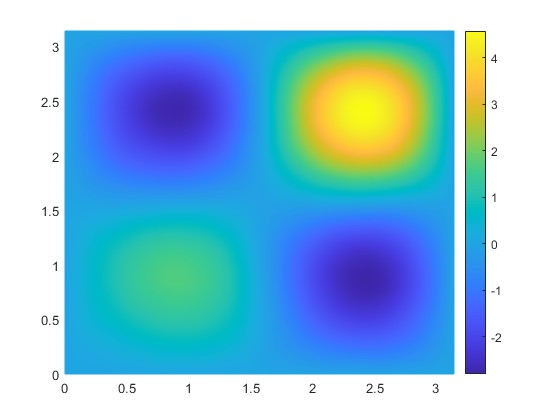}}	
		\subfigure[Recovered result for 50$\%$ noise]{\includegraphics[width=0.32\linewidth]{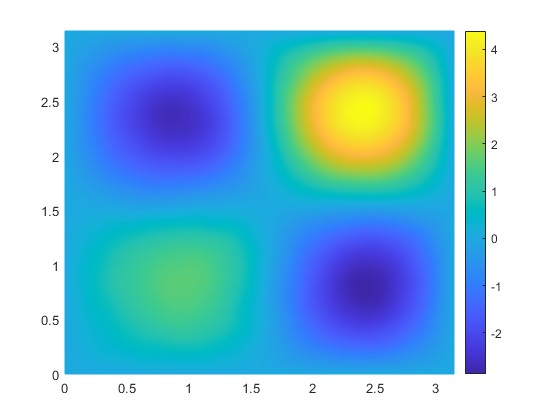}}	
		\caption{ Robustness of the adjoint POD against the noise for $f=sin(2x)sin(2y)e^{\frac{x+y}{\pi}}$. }\label{sin2exp-noise-source}
\end{figure}

 \subsection{Examples for backward problem}
	
 In this subsection, we will apply the new POD method to recover the initial term $g$ as discussed in Section \ref{sec:backward}. We will collect the data at the time $T=0.05$.
	

       \textbf{Example 4.4} We will demonstrate the importance of selecting the appropriate POD basis functions. For the same source term, we apply different right-hand sides of equation \eqref{backward} to derive the POD basis functions. Then, we solve the backward problem using different POD basis functions. 
       
       Figure \ref{right POD basis-ini} illustrates the corresponding results. The true source term is $\sin(2x)\sin(2y)e^{\frac{x+y}{\pi}}$, and its surface plot is shown in Figure \ref{right POD basis-ini} (a). Figure \ref{right POD basis-ini} (b) presents the reconstruction result using our adjoint POD method proposed in this paper, which demonstrates the effectiveness and efficiency of our new POD method in recovering the source term. Figure \ref{right POD basis-ini} (c) shows the result when we use an incorrect right-hand side to generate the POD basis functions. In this case, we use $\sin(x)\sin(y)$ as the right-hand side in \eqref{backward} to generate the POD basis functions. Figure \ref{right POD basis-ini} (d) displays the result when we use an A-shaped function as the right-hand side to generate the POD basis functions. It can be observed that Figure \ref{right POD basis-ini} (b) provides a good reconstruction, while Figure \ref{right POD basis-ini} (c) and Figure \ref{right POD basis-ini} (d) yield incorrect results. Particularly in Figure \ref{right POD basis-ini} (d), the recovered image is entirely different from the exact source term.

       In this example, we demonstrate the importance of choosing the correct basis to solve the inverse problem. Using an inappropriate set of basis functions may lead to unsatisfactory results. Our proposed adjoint POD provides a set of suitable basis functions. In the upcoming examples, we will compare our adjoint POD basis functions with the original true POD basis functions to verify Theorem \ref{theorem-pod-app-ini}.

	\begin{figure}[htb]
		\centering
		\subfigure[Exact initial term]{\includegraphics[width=0.32\linewidth]{Figures/source-sin2exp-true.jpg}}
        \subfigure[POD method proposed in this paper]{\includegraphics[width=0.32\linewidth]{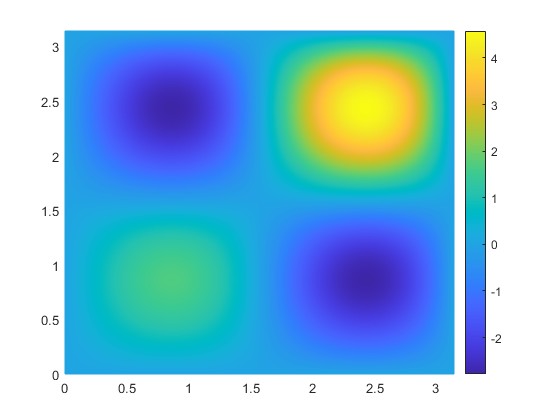}}\\
		\subfigure[Using the POD basis generated with the right hand side $sin(x)sin(y)$]{\includegraphics[width=0.32\linewidth]{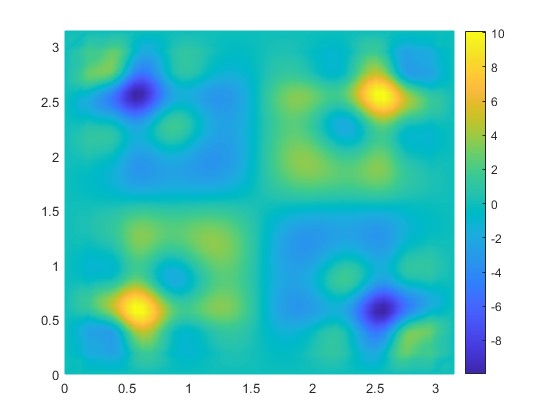}}	
		\subfigure[Using the POD basis generated with the right hand side of A-shaped function]{\includegraphics[width=0.32\linewidth]{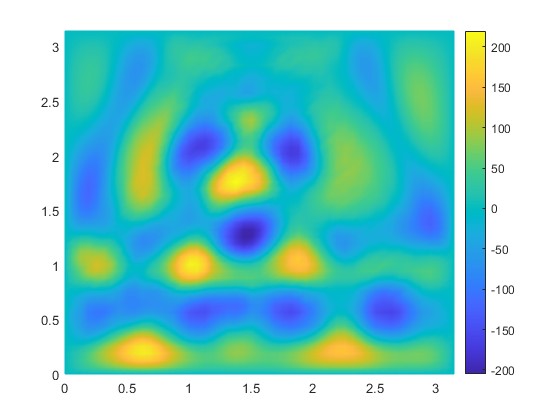}}
		\caption{The importance of choice of POD basis }\label{right POD basis-ini}
	\end{figure}

   \textbf{Example 4.5} In the following two examples, we first use the true source term as the right-hand side in Eq. \eqref{backward} to generate the POD basis functions. Then, we generate the POD basis functions using our proposed adjoint POD method. To verify Theorem \ref{theorem-pod-app-ini}, we will plot two sets of basis functions to determine if they are closely related. Figure \ref{fig:ini-eg_sin2exp-basis} shows the results using the exact source term $g^* = \sin(2x) \sin(2y) e^{\frac{x+y}{\pi}}$. Figures \ref{fig:ini-eg_sin2exp-basis} (b) and \ref{fig:ini-eg_sin2exp-basis} (d) demonstrate that both the traditional POD and our adjoint POD methods work well in recovering the true initial term. However, the difference is that our adjoint POD method does not require prior knowledge of the exact initial term, while the traditional POD method does, which is known as the inverse crime. Figures \ref{fig:ini-eg_sin2exp-basis} (c) and \ref{fig:ini-eg_sin2exp-basis} (e) display the basis functions of each method. We can conclude that the basis functions of our adjoint POD method are very close to the traditional one, and we obtain it solely from the measured data.

    Figure \ref{fig:ini-eg_A-basis} shows the results using the exact initial term $g^*$ of the A-shaped function. This example also illustrates why the POD method is more efficient in solving inverse problems compared to the finite element method. Figures \ref{fig:ini-eg_A-basis} (c) and \ref{fig:ini-eg_A-basis} (e) reveal that both the traditional POD method and our adjoint POD method's basis functions contain critical information about the exact initial term we aim to recover. In contrast, the finite element method's basis lacks any priori information about the true function we need to recover.

	\begin{figure}[htb]
		\flushright
		\subfigure[Exact initial term]{\includegraphics[width=0.32\linewidth]{Figures/source-sin2exp-true.jpg}}
		\subfigure[Recovered result by traditional POD with exact initial term]{\includegraphics[width=0.32\linewidth]{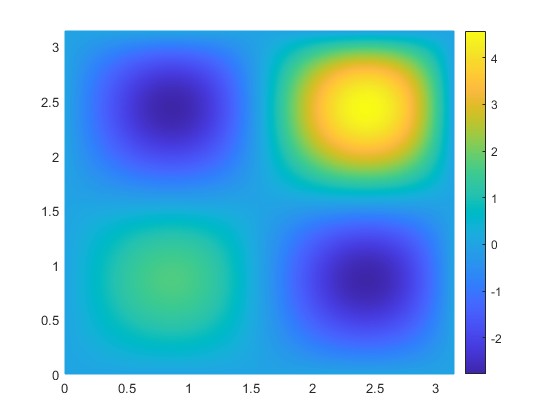}}	
		\subfigure[Basis of the traditional POD]{\includegraphics[width=0.32\linewidth]{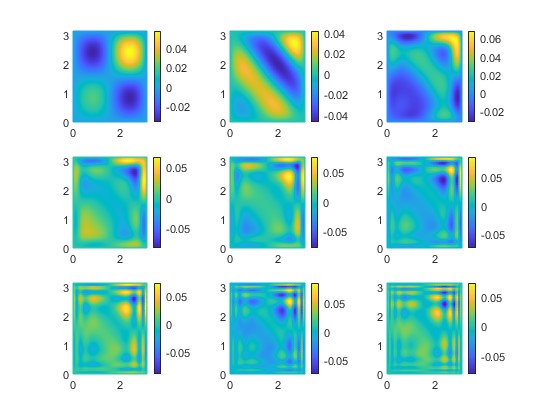}}
		\subfigure[Result by the adjoint POD]{\includegraphics[width=0.32\linewidth]{Figures/ini-sin2exp-newpod.jpg}}	\subfigure[Basis of the adjoint POD]{\includegraphics[width=0.32\linewidth]{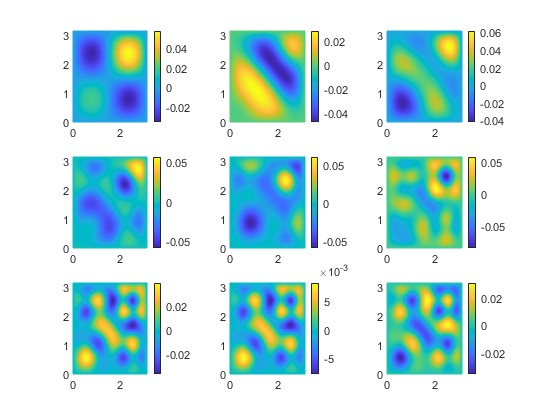}}
		\caption{ Comparison of basis between traditional POD and the adjoint POD for $f=sin(2x)sin(2y)e^{\frac{x+y}{\pi}}$. }
		\label{fig:ini-eg_sin2exp-basis}
	\end{figure}

	\begin{figure}[htb]
		\flushright
		\subfigure[Exact initial term]{\includegraphics[width=0.32\linewidth]{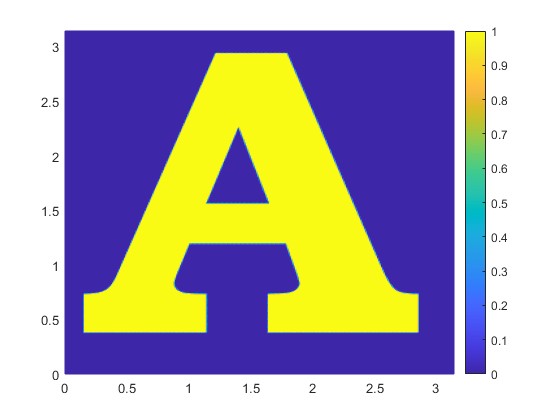}}
		\subfigure[Recovered result by traditional POD with exact initial term]{\includegraphics[width=0.32\linewidth]{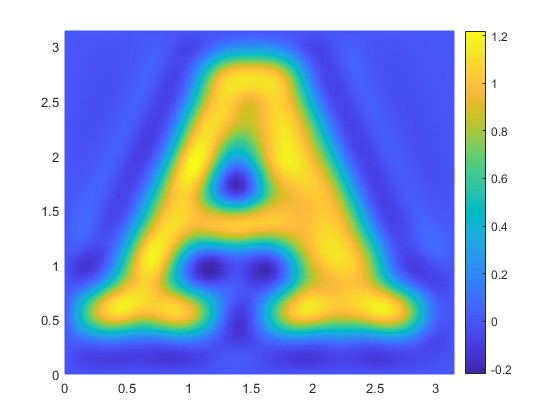}}	
		\subfigure[Basis of the traditional POD]{\includegraphics[width=0.32\linewidth]{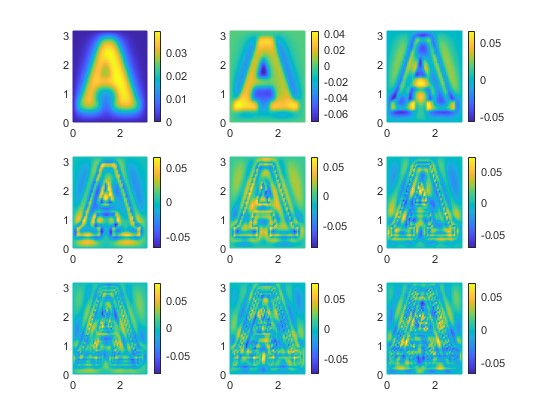}}
		\subfigure[Result by the adjoint POD]{\includegraphics[width=0.32\linewidth]{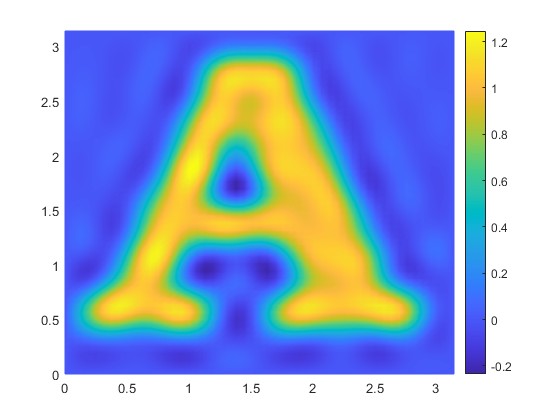}}	\subfigure[Basis of the adjoint POD]{\includegraphics[width=0.32\linewidth]{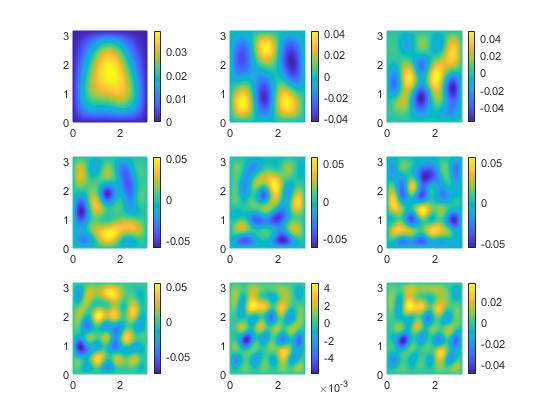}}
		\caption{  Comparison of basis between traditional POD and  the adjoint POD for $f^*$ of A shaped function.}
		\label{fig:ini-eg_A-basis}
	\end{figure}


 In the above cases, all the measured data are noise-free. We will now test the denoising method discussed in Section \ref{sec:convergence-ini} by examining very challenging cases with noise levels ranging from $10\%$  to $50\%$.


     \textbf{Example 4.6} In this example, we will test the robustness of the adjoint POD method against noise. We take the measurement data as $m^n_i = u(d_i, T) + \sigma e_i$, for $i = 1, \cdots, n$, where $d_i$s represent the positions inside of $\Omega$, and $\{e_i\}^n_{i=1}$ are independent standard normal random variables. We will use 2500 positions $d_i$ uniformly distributed over the domain $\Omega$. Figure \ref{sin2-noise-ini} demonstrates that our method is robust even in the presence of big noise. Remarkably, even with $50\%$ noise, when the measured data is completely obscured by noise as shown in Figure \ref{sin2-noise-ini} (c), we can still recover the source term, as seen in Figure \ref{sin2-noise-ini} (f).

 \begin{figure}[htb]
		\centering
		\subfigure[Exact initial term]{\includegraphics[width=0.32\linewidth]{Figures/source-sin2-true.jpg}}
        \subfigure[Measured data with 25$\%$ noise]{\includegraphics[width=0.32\linewidth]{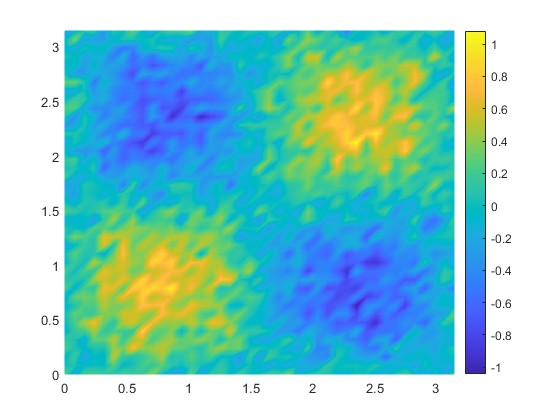}}	
        \subfigure[Measured data with 50$\%$ noise]{\includegraphics[width=0.32\linewidth]{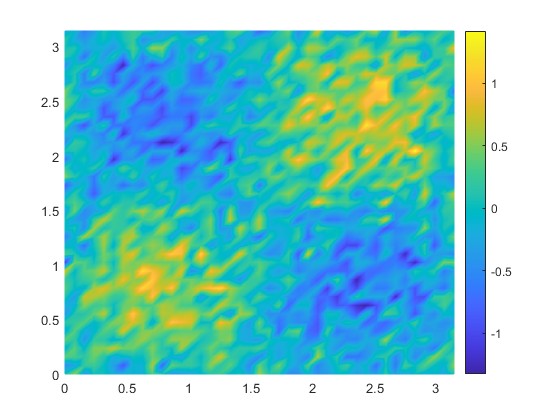}}
		\subfigure[Recovered result for 10$\%$ noise]{\includegraphics[width=0.32\linewidth]{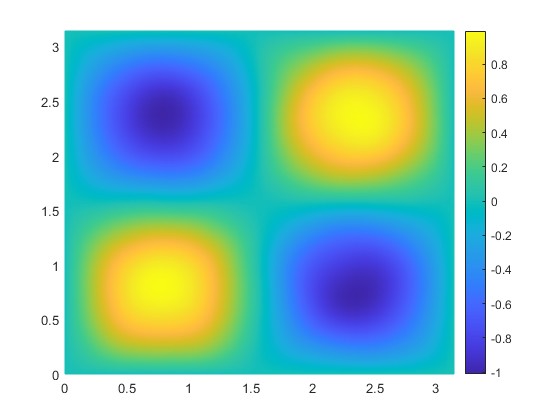}}
        \subfigure[Recovered result for 25$\%$ noise]{\includegraphics[width=0.32\linewidth]{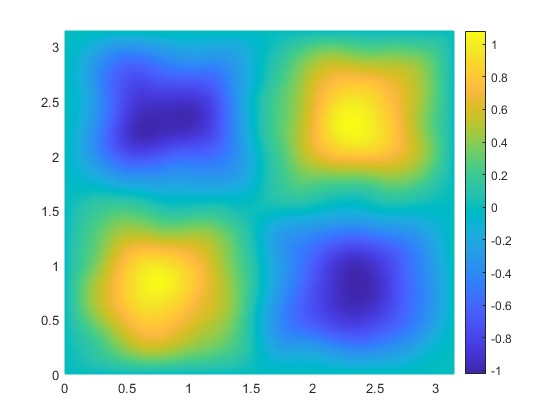}}	
		\subfigure[Recovered result for 50$\%$ noise]{\includegraphics[width=0.32\linewidth]{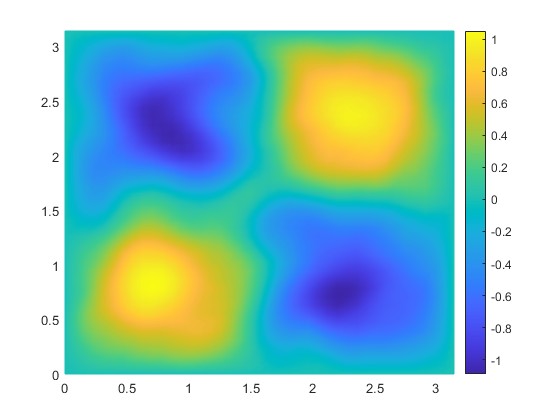}}	
		\caption{ Robustness of the adjoint POD against the noise for $g=sin(2x)sin(2y)$. }\label{sin2-noise-ini}
\end{figure}

    \textbf{Example 4.7}  Finally, we study a more interesting case. While the inverse source problem and the backward problem are two distinct problems, we have observed from the previous numerical examples that they share some commonalities. Specifically, the POD basis for both problems contains critical information about the functions one wants to recover. As a result, we will employ the POD basis functions derived from the inverse source problem to solve the backward problem. Please refer to Figure \ref{fig:ini-eg_A_using source-basis} for the reconstruction results.

	\begin{figure}[htb]
		\flushright
		\subfigure[Exact initial term]{\includegraphics[width=0.32\linewidth]{Figures/source-A-true.jpg}}
		\subfigure[Recovered result using the POD basis from backward problem]{\includegraphics[width=0.32\linewidth]{Figures/ini-A-newpod.jpg}}
		\subfigure[Recovered result using the POD basis from inverse source problem]{\includegraphics[width=0.32\linewidth]{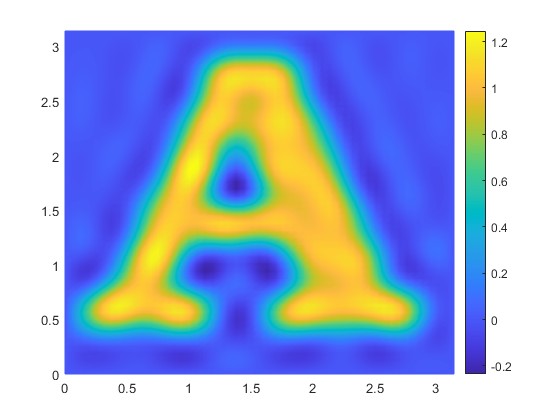}}
		\caption{ Solve the backward problem using the basis from inverse source problem. }
		\label{fig:ini-eg_A_using-source basis}
	\end{figure}

\section{Conclusion}
We have developed a data-driven and model-based approach for solving parabolic inverse source problems with uncertain data. The key idea is to exploit the model-based intrinsic low-dimensional structure of the underlying parabolic PDEs and construct data-based POD basis functions to achieve significant dimension reduction in the solution space. Equipped with the POD basis functions, we develop a fast algorithm that can compute the optimization problem in the inverse source problems. Hence, we obtain an effective data-driven and model-based approach for the inverse source problems and overcome the typical computational bottleneck of FEM in solving these problems.
Under a weak assumption on the regularity of the solution, we provide the convergence analysis of our POD algorithm in solving the forward parabolic PDEs and thus obtain the error estimate of the POD algorithm for the parabolic inverse source problems. Finally, we carry out numerical experiments to demonstrate the accuracy and efficiency of the proposed method. We also study other issues of the POD algorithm, such as the dependence of the error on the mesh size, the regularization parameter in the least-squares regularized minimization problems, and the number of POD basis functions. Through numerical results, we find that our POD algorithm provides significant computational savings over the FEM while yielding as good approximations as the FEM.  We expect an even better performance of efficiency can be obtained for 3D problems, which will be studied in our future works.

\section*{Acknowledgement}

\noindent
The research of W. Zhang is supported by the National Natural Science Foundation of China No. 12371423 and 12241104. The research of Z. Zhang is supported by Hong Kong RGC grant project 17307921, National Natural Science Foundation of China No. 12171406, Seed Funding for Strategic Interdisciplinary Research Scheme 2021/22 (HKU), and Seed Funding from the HKU-TCL Joint Research Centre for Artificial Intelligence. 

\appendix
\section{Proper orthogonal decomposition (POD) method}\label{appendix-POD}
	\noindent
Assuming that $u \in H^1_0(\Omega)$ is the solution to the weak formulation of the parabolic equation \eqref{parabolic-equation}, the construction of POD basis functions requires solution snapshots. These solution snapshots can be obtained by the appropriate technological means related to a specific application, such as experimental data or numerical methods.
 
	Given a set of solutions at different time instances $\big\{ u(\cdot, t_0), u(\cdot, t_1), \ldots, u(\cdot, t_M) \big\}$, where $t_k = k \Delta t$ with $\Delta t = \frac{T}{M}$ and $k = 0, \ldots, M$, we first obtain the solution snapshots $\{ y_1, \ldots, y_{M + 1},$ $ y_{M + 2}, \ldots, y_{2M + 1} \}$, where 
	$y_k = u(\cdot, t_{k - 1})$, $k = 1, \ldots, M + 1$, and $y_k = \overline{\partial} u(\cdot, t_{k - M - 1})$, $k = M + 2, \ldots, 2m + 1$ with $\overline{\partial} u(\cdot, t_{k}) = \frac{u(\cdot, t_{k}) - u(\cdot, t_{k - 1})}{\Delta t}$, $k = 1, \ldots, M$.
	 
	Then, the POD basis functions $\{ \psi_k \}_{k = 1}^{{N_{\text{pod}}}}$ are constructed by minimizing the following projection error: 
	\begin{align}
		\frac{1}{2m+1}\Big( \sum_{j = 0}^M \big\| u(t_j) - \sum_{k = 1}^{{N_{\text{pod}}}} (u(t_j), \psi_k)_{L^2(\Omega)} \psi_k \big\|_{L^2(\Omega)}^2 \\
		+	\sum_{j = 1}^M \big\| \overline{\partial} u(t_j) - \sum_{k = 1}^{{N_{\text{pod}}}} (\overline{\partial} u(t_j), \psi_k)_{L^2(\Omega)} \psi_k \big\|_{L^2(\Omega)}^2 \Big)
	\end{align}
	subject to the constraints that $\big(\psi_{k_1}(\cdot),\psi_{k_2}(\cdot)\big)_{L^2(\Omega)}=\delta_{k_1k_2}$, $1\leq k_1,k_2 \leq N_{\text{pod}}$, where $\delta_{k_1k_2}=1$ if $k_1=k_2$, otherwise $\delta_{k_1k_2}=0$. 
	Here, we use $N_{\text{pod}}$ to denote the number of POD basis functions that will be extracted from solution snapshots. 
	
	Let ${V_{\text{pod}}} = \text{span}\{ \psi_1, \ldots, \psi_{N_{\text{pod}}} \}$ denote the finite-dimensional space spanned by the POD basis functions. Using the method of snapshot proposed by Sirovich \cite{Sirovich:1987}, we know that the minimizing problem can be reduced to the following eigenvalue problem:
	
	\begin{equation}\label{eigenvalueproblemPOD}
		Kv = \mu v,
	\end{equation}
	where the correlation matrix $K$ is computed from the solution snapshots $\{ y_1, y_2,\ldots, y_{2M + 1} \}$ with entries $K_{ij} = (y_i, y_j)_{L^2(\Omega)}$, $i,j = 1, \ldots, 2M + 1$, and $K$ is symmetric and semi-positive definite. 
	We sort the eigenvalues in a decreasing order as $\lambda_1 \geq \lambda_2 \geq ... \geq \lambda_{2m + 1}$ and
	the corresponding eigenvectors are denoted by $v_k$, $k=1,...,2M + 1$. It can be shown that if the POD basis functions are constructed by
	
	\begin{equation}\label{PODbasisMethodOfSnapshot}
		\varphi_{k}(\cdot) = \frac{1}{\sqrt{\lambda_k}}\sum_{j=1}^{2M + 1}(v_k)_j u(\cdot,t_j), \quad 1\leq k \leq N_{\text{pod}},
	\end{equation}
	where $(v_k)_j$ is the $j$-th  component of the eigenvector $v_k$, they minimize the projection error. 
	
 	The approximation error for the POD method has been studied extensively in the literature, particularly in the works \cite{holmes:98} and \cite{Willcox2015PODsurvey}. 
  
	\begin{prop}[Sec. 3.3.2, \cite{holmes:98} or p. 502, \cite{Willcox2015PODsurvey}]\label{Prop_PODError}
		Let $\lambda_1 \geq \lambda_2 \geq ... \geq \lambda_{2M + 1} \geq 0$ denote the non-negative eigenvalues of $K$ in the eigenvalue problem \eqref{eigenvalueproblemPOD}. Then, $\{\psi_k\}_{k=1}^{N_{\text{pod}}}$ constructed according to the method of snapshots \eqref{PODbasisMethodOfSnapshot} is the set of POD basis functions, and we have the following error formula:
		\begin{equation}\label{SnapshotOfSolutions}		
			\frac{\sum_{i=1}^{2M + 1}\left|\left|\tilde y_i - \sum_{k=1}^{{N_{\text{pod}}}}\big(\tilde y_i,\psi_k(\cdot)\big)_{L^2(\Omega}\psi_k(\cdot)\right|\right|_{L^2(\Omega)}^{2}}{		 \sum_{i=1}^{2M+1}\left|\left|\tilde y_i\right|\right|_{L^2(\Omega)}^{2}} 
			= \frac{\sum_{k={N_{\text{pod}}}+1}^{2M + 1}  \lambda_k}{\sum_{k=1}^{2M + 1}  \lambda_k},
		\end{equation} 
		where the number $N_{\text{pod}}$ is determined according to the decay of the ratio $\rho=\frac{\sum_{k={N_{\text{pod}}}+1}^{2M + 1}  \lambda_k}{\sum_{k=1}^{2M + 1}\lambda_k}$. 
	\end{prop}

\end{document}